\documentclass[oneside,reqno]{amsart}

\usepackage{amssymb}
\usepackage{graphicx}

\usepackage{url}            
\usepackage{booktabs}       
\usepackage{amsfonts}       
\usepackage{nicefrac}       
\usepackage{microtype}      
\usepackage{amsmath}
\usepackage{bm}
\usepackage{dsfont}
\usepackage{color,soul}
\usepackage{mdframed,xcolor}
\usepackage[hidelinks]{hyperref}   
\usepackage[round]{natbib}
\usepackage{tabularx}
\usepackage{graphicx}
\usepackage{caption}
\usepackage{subcaption}
\usepackage{amsaddr}

\newtheorem{theorem}{Theorem}[section]
\newtheorem{lemma}[theorem]{Lemma}

\theoremstyle{definition}

\newtheorem{example}[theorem]{Example}

\theoremstyle{remark}

\newcommand{\R}{\mathbb{R}}
\newcommand{\XX}{\mathcal{X}}

\newcommand{\E}{\mathbb{E}}

\newcommand{\Var}{\operatorname{Var}}

\newcommand*\diff{\mathop{}\!\mathrm{d}}

\newcommand{\1}{\mathds{1}}

\DeclareMathOperator{\eff}{eff}

\newcommand{\X}{\boldsymbol{X}}
\newcommand{\B}{\boldsymbol{\beta}}

\newcommand{\f}{{\bf f}}
\newcommand{\M}{{\bf M}}
\renewcommand{\E}{\operatorname{E}}
\newcommand{\Prob}{\operatorname{P}}
\newcommand{\trace}{\operatorname{trace}}
\newtheorem{corollary}[theorem]{Corollary}

\begin{document}
	
	\title[Optimal Subsampling Design for Polynomial Regression]{Optimal Subsampling Design for Polynomial Regression in one Covariate}
	
	
	\author{Torsten Reuter}
	\thanks{Corresponding author: Torsten Reuter.}
	\address{Otto von Guericke University Magdeburg. Universitätsplatz 2, 39106 Magdeburg,
		Germany}
	\curraddr{}
	\email{torsten.reuter@ovgu.de}
	
	\author{Rainer Schwabe}
	\address{Otto von Guericke University Magdeburg. Universitätsplatz 2, 39106 Magdeburg,
		Germany}
	\curraddr{}
	\email{rainer.schwabe@ovgu.de}
	
	\date{}
	

	\begin{abstract}
	Improvements in technology lead to increasing availability of large data sets which makes the need for data reduction and informative subsamples ever more important. 
	In this paper we construct $ D $-optimal subsampling designs for polynomial regression in one covariate for invariant distributions of the covariate. 
	We study quadratic regression more closely for specific distributions. 
	In particular we make statements on the shape of the resulting optimal subsampling designs and the effect of the subsample size on the design. 
	To illustrate the advantage of the optimal subsampling designs we examine the efficiency of uniform random subsampling. 
	\end{abstract}
	
	\maketitle
	
\section{Introduction}
\label{sec:introduction}

Data Reduction is a major challenge as technological advances have led to a massive increase in data collection to a point where traditional statistical methods fail or computing power can not keep up. In this case we speak of big data. 
We typically differentiate between the case where the number of  is much larger than the number of observations and the case where the massive amount of observations is the problem. 
The first case is well studied, most notably by \cite{tibshirani1996regression} introducing LASSO, which utilizes $ \ell_{1} $ penalization to find sparse parameter vectors, thus fusing subset selection and ridge regression. 
The second case, often referred to as massive data, can be tackled in two ways. 
Firstly in a probabilistic fashion, creating random subsamples in a non-uniform manner. 
Prominent studies include \cite{drineas2006sampling}, \cite{mahoney2011random} and \cite{ma2014statistical}. 
They present subsampling methods for linear regression models called algorithmic leveraging that sample according to probabilities based on the normalized statistical leverage scores of the covariate matrix. 
More recently \cite{derezinski2018reverse} study volume sampling, where subdata is chosen proportional to the squared volume of the parallelepiped spanned by its observations. 
Conversely to these probabilistic methods one can select subdata by applying deterministic rules. 
\cite{shi2021model} present such a method, that maximizes the minimal distance between two observations in the subdata.
\cite{wang2021orthogonal} propose orthogonal subsampling inspired by orthogonal arrays. 
Most prominently, \cite{wang2019information} introduce the information-based optimal subdata selection (IBOSS) to tackle big data linear regression in a deterministic fashion based on $ D $-optimality.  

In this paper we study $ D $-optimal subsampling designs for polynomial regression in one covariate, where the goal is to select a percentage $ \alpha $ of the full data that maximizes the determinant of the information matrix. 
For the conventional study of approximate designs in this setting we refer to \cite{gaffke1996approximate}. \cite{heiligers1992invariant} consider specifically cubic regression on a ball. 
We consider $ D $-optimal designs with measure $ \alpha $ that are bounded from above by the distribution of the known covariate. 
Such directly bounded designs were first studied by \cite{wynn1977optimum} and \cite{fedorov1989optimal}.
\cite{pronzato2004minimax} considers this setting using a form of the subsampling design standardized to one and bounded by $\alpha^{-1}$ times the distribution of the covariates. 
More recently, \cite{pronzato2021sequential} studies the same in the context of sequential subsampling. For the characterization of the optimal subsampling designs we make use of an equivalence theorem by \cite{sahm2001note}. 
This equivalence theorem enables us to construct such subsampling designs for various settings of the distributional assumptions on the covariate. 
Here we will only look at distributions of the covariate that are invariant to a sign change, i.e. symmetric about the vertical axis. 
We discuss the shape of $ D $-optimal subsampling subsampling designs for polynomial regression of degree $ q $ first. 
We then study quadratic regression under several distributional assumptions more closely, after showing two examples for simple linear regression. 
In particular we take a look at the percentage of mass of the optimal subsampling design on the outer intervals compared to the inner one, which changes drastically given the distribution of the covariate, particularly for heavy-tailed distributions. 
In addition we examine the efficiency of uniform random subsampling to illustrate the advantage of the optimal subsampling designs. 
All numerical results are obtained by the Newton method 
implemented in the \textbf{\textsf{R}} package \textit{nleqslv} by \cite{nleqslv}.
All relevant \textbf{\textsf{R}} scripts are available on a GitHub repository
\url{https://github.com/TorstenReuter/polynomial_regression_in_one_covariate}.

The rest of this paper is organized as follows. 
In Section~\ref{sec:model} 
we specify the polynomial model. 
In Section~\ref{sec:design} we
introduce the concept of continuous subsampling designs
and give characterizations for optimization.
In Sections~\ref{sec:linear} and \ref{sec:quadratic} 
we present optimal subsampling designs 
in the case of linear and quadratic regression, respectively,
for various classes of distributions of the covariate.
Section~\ref{sec:efficiency} contains some efficiency considerations
showing the strength of improvement of
the performance of the optimal subsampling design 
compared to random subsampling.
The paper concludes with a discussion  
in Section~\ref{sec:discussion}.
Proofs are deferred to an Appendix.

\section{Model Specification}
\label{sec:model}

We consider the situation of pairs $(x_i,y_i)$ of data, 
where $y_i$ is the value of the response variable $Y_i$
and $x_i$ is the value of a single covariate $X_i$ for unit $i = 1,\dots,n$,
for very large numbers of units $ n $.
We assume that the dependence of the response on the covariate
is given by a polynomial regression model
\begin{equation*}
	Y_{i} =\beta_0 + \beta_1 X_{i} + \beta_2 X_{i}^2 + \dots + \beta_q X_{i}^q + \varepsilon_{i} 
\end{equation*}
with independent, homoscedastic random errors $\varepsilon_i$ having zero mean 
($\E(\varepsilon_i)=0$, $\Var(\varepsilon_i)=\sigma_\varepsilon^2>0$).
The largest exponent $q \geq 1$ denotes the degree of the polynomial regression,
and $p=q+1$ is the number of regression parameters 
$\beta_0, \dots,\beta_q$ to be estimated,
where, for each $k= 1,\dots,q$, 
the parameter $\beta_k$ is the coefficient 
for the $k$th monomial $x^k$, and  $\beta_0$ denotes the intercept.
For example, for $q=1$, we have ordinary linear regression,
$Y_{i} =\beta_0 + \beta_1 X_{i} + \varepsilon_{i}$, 
with $p=2$ parameters $\beta_0$ (intercept) and $\beta_1$ (slope)
and, for $q=2$, we have quadratic regression,
$Y_{i} =\beta_0 + \beta_1 X_{i} + \beta_2 X_{i}^2 + \varepsilon_{i}$,
with $p=3$ and an additional curvature parameter $\beta_2$.
Further, we assume that the units of the covariate $X_i$ are identically distributed
and that all $X_i$ and random errors $\varepsilon_{i^\prime}$
are independent.

For notational convenience, we write the polynomial regression 
as a general linear model
\begin{equation*}
	Y_{i} = \f(X_{i})^{\top}\B + \varepsilon_{i} \, , 
\end{equation*}
where $ \f(x) = (1, x, \dots, x^{q})^{\top}$ is the $p$-dimensional vector
of regression functions and $ \B = (\beta_0, \beta_1,\dots,\beta_q)^\top$
is the $p$-dimensional vector of regression parameters.

\section{Subsampling Design}
\label{sec:design}

We are faced with the problem that the responses $ Y_{i} $ 
are expensive or difficult to observe
while the values $x_i$ of all units $X_i$ of the covariate are available.
To overcome this problem, we consider the situation that 
the responses $ Y_{i} $ will be observed only for a certain 
percentage $ \alpha $ of the units ($0<\alpha<1$) 
and that these units will be selected on the basis of the knowledge 
of the values $ x_{i} $ of the covariate for all units. 
As an alternative motivation, we can consider a situation 
where all pairs $ (x_{i}, y_{i}) $ are available but parameter estimation 
is computationally feasible only on a percentage $ \alpha $ of the data. 
In either case we want to find the subsample of pairs $ (x_{i}, y_{i}) $ 
that yields the most precise estimation of the parameter vector $ \B $. 

To obtain analytical results, the covariate $ X_{i} $ 
is supposed to have a continuous distribution with density $ f_{X}(x) $,
and we assume that the distribution of the covariate is known. 
The aim is to find a subsample of this distribution that covers a percentage
$\alpha$ of the distribution and that contains the most information. 
For this, we will consider continuous designs $ \xi $ 
as measures of mass $\alpha$
on $\R$ with density $f_{\xi}(x)$ bounded by the density $f_X(x)$ 
of the covariate $X_i$ such that $ \int f_{\xi}(x)\diff x = \alpha $ and
 $  f_{\xi}(x) \leq f_X(x)$ for all $x \in \R$. 
A subsample can then be generated according to such a continuous design
by accepting units $i$ with probability $f_\xi(x_i)/f_X(x_i)$.

For a continuous design $\xi$, the information matrix $ \M(\xi)$ is defined as  
$ \M(\xi) = \int \f(x)\f(x)^\top  f_\xi(x) \diff x$.
In the present polynomial setup, 
$ \M(\xi) = \left(m_{j + j^\prime}(\xi)\right)_{j = 0, \dots, q}^{j^\prime = 0, \dots, q}$,
where $m_k = \int x^k  f_\xi(x) \diff x$ is the $k$th moment
associated with the design $\xi$.
Thus, it has to be required 
that the distribution of $X_i$ has a finite moment $\E(X_i^{2q})$ of order $2q$
in order to guarantee that all entries in the information matrix $\M(\xi)$ exist
for all continuous designs $\xi$ 
for which the density $f_\xi(x)$ is bounded by $f_X(x)$.

The information matrix $\M(\xi)$ measures the performance of the design $\xi$
in the sense that the covariance matrix of the least squares estimator
$\hat\B$ based on a subsample according to the design $\xi$ is proportional 
to the inverse $\M(\xi)^{-1}$ of the information matrix $\M(\xi)$
or, more precisely,  $\sqrt{\alpha n}(\hat\B-\B)$ is normally distributed 
with mean zero and covariance matrix $\sigma_\varepsilon^2\M(\xi)^{-1}$,
at least asymptotically.
Note that for continuous designs $\xi$ the information matrix $\M(\xi)$
is always of full rank and, hence, the inverse $\M(\xi)^{-1}$ exists.
Based on the relation to the covariance matrix,
it is desirable to maximize the information matrix $\M(\xi)$. 
However, as well-known in design optimization, 
maximization of the information matrix 
cannot be achieved uniformly with respect to the Loewner ordering
of positive-definiteness. Thus, commonly, a design criterion
which is a real valued functional of
the information matrix $\M(\xi)$ will be maximized, instead.
We will focus here on the most popular design criterion in applications,
the $D$-criterion, in its common form $\log(\det(\M(\xi)))$ to be maximized.
Maximization of the $D$-criterion can be interpreted in terms of the
covariance matrix to be the same as minimizing the volume
of the confidence ellipsoid for the whole parameter vector $\B$
based on the least squares estimator 
or, equivalently, minimizing the volume of the 
acceptance region for a Wald test on the whole model.
The subsampling design $\xi^{*}$ that maximizes the $D$-criterion 
$\log(\det(\M(\xi)))$ will be called $D$-optimal, and its density is denoted 
by  $ f_{\xi^{*}}(x) $. 

To obtain $D$-optimal subsampling designs, we will make use of standard techniques
coming from constrained convex optimization and symmetrization.
For convex optimization we employ the directional derivative
\begin{equation*}
	F_{D}(\xi, \eta) = \lim_{\epsilon \to 0^+} \frac{1}{\epsilon}
	\left(\log(\det(\M((1 - \epsilon) \xi + \epsilon \eta))) - \log(\det(\M(\xi)))\right)
\end{equation*}
of the $D$-criterion at a design $\xi$ 
with non-singular information matrix $\M(\xi)$ 
in the direction of a design $\eta$,
where we allow here $\eta$ to be a general design of mass $\alpha$
that has not necessarily a density bounded by $f_{X}(x)$.
In particular, $\eta = \xi_x$ may be a one-point design which assigns all mass $\alpha$ 
to a single setting $x$ in $\R$.
Evaluating of the directional derivative yields
$F_{D}(\xi, \eta) = p - \trace(\M(\xi)^{-1}\M(\eta))$
\cite[compare][Example 3.8]{silvey1980optimal}
which reduces to  
$F_{D}(\xi, \xi_x) = p - \alpha \f(x)^\top \M(\xi)^{-1} \f(x)$
for a one-point design $\eta = \xi_x$.
Equivalently, for one-point designs $\eta = \xi_x$, we may consider
the sensitivity function
$\psi(x, \xi) = \alpha \f(x)^\top \M(\xi)^{-1} \f(x)$
which incorporates the essential part of the directional derivative
($\psi(x, \xi) =  p - F_{D}(\xi, \xi_x)$).
For the characterization of the $D$-optimal continuous subsampling design,
the constrained equivalence theorem
under Kuhn-Tucker conditions 
\cite[see][Corollary~1~(c)]{sahm2001note}
can be reformulated in terms of the sensitivity function and applied 
to our case of polynomial regression. 

\begin{theorem}
	\label{th:opt-design-degree-q}
	In polynomial regression of degree $q$ with density $f_X(x)$ of the covariate $X_i$, 
	the subsampling design $\xi^*$ with support $ \XX^{*} $ is $D$-optimal if and only if
	there exist a threshold $s^*$ 
	and settings $a_1 > \dots > a_{2r}$ for some $r$ ($1 \leq r \leq q$) such that
	\begin{itemize}
		\item[(i)]	
			the $D$-optimal subsampling design $\xi^{*}$ is given by
			\begin{equation*}
				\label{eq:opt-dens-degree-q}
				f_{\xi^{*}}(x) = 
				\left\{
					\begin{array}{ll}
						f_{X}(x) & \mbox{if } x \in \XX^* 
						\\
						0 & \mbox{otherwise}
					\end{array}
				\right. \,
			\end{equation*}
		\item[(ii)] $\psi(x, \xi^*) \geq s^*$ for $x \in \XX^*$, and
		\item[(iii)] $\psi(x, \xi^*) < s^*$ for $x \not\in \XX^*$,
	\end{itemize}	
	where $\XX^* = \bigcup_{k=0}^r \mathcal{I}_k$
	and $\mathcal{I}_0 = [a_{1}, \infty)$, 
	$\mathcal{I}_r = (-\infty, a_{2r}]$, and
	$\mathcal{I}_k = [a_{2k + 1}, a_{2k}]$, for $k = 1, \dots, r - 1$,
	are mutually disjoint intervals.
\end{theorem}

The density $f_{\xi^*}(x) = f_X(x) \1_{\XX^*}(x)
 = \sum_{k=0}^r f_X(x) \1_{\mathcal{I}_k}(x)$
of the $D$-optimal subsampling design $\xi^*$ is concentrated on, at most,
$q + 1$ intervals $\mathcal{I}_k$, 
where $\1_{A}(x)$ denotes the indicator function on the set $A$,
i.\,e.\ $\1_A(x) = 1$ for $x \in A$, and $\1_A(x) = 0$ otherwise.
The density $f_{\xi^*}(x)$ has a $0$-$1$-property such that it 
is either equal to the density $f_X(x)$ of the covariate (on $\XX^*$)
or equal to $0$ (on the complement of $\XX^*$).
Thus, the generation of a subsample according to the optimal 
continuous subsampling design $\xi^*$ can be implemented easily
by accepting all units $i$ for which the value $x_i$ of the covariate 
is in $\XX^*$ and rejecting all other units with $x_i \not\in \XX^*$.
The threshold $s^*$ can be interpreted as the $(1- \alpha)$-quantile 
of the distribution of the sensitivity function $\psi(X_i,\xi^*)$
as a function of the random variable $X_i$
\cite[see][]{pronzato2021sequential}.

A further general concept to be used is equivariance.
This can be employed to transform the $D$-optimal subsampling design simultaneously 
with a transformation of the distribution of the covariate.
More precisely, the location-scale transformation $Z_i = \sigma X_i + \mu$
of the covariate and its distribution is conformable with
the regression function $\f(x)$ in polynomial regression, 
and the $D$-criterion is equivariant with respect to such transformations.

\begin{theorem}
	\label{th:equivariant}
	Let $f_{\xi^*}(x)$ be the density for a $D$-optimal subsampling design $\xi^*$ 
	for covariate $X_i$ with density $f_X(x)$.
	Then $f_{\zeta^*}(z) = \frac{1}{\sigma} f_{\xi^*}(\frac{z - \mu}{\sigma})$ 
	is the density for a $D$-optimal subsampling design $\zeta^*$ 
	for covariate $Z_i = \sigma X_i + \mu$ with density 
	$f_Z(z) = \frac{1}{\sigma} f_{X}(\frac{z - \mu}{\sigma})$.
\end{theorem}

As a consequence, also the optimal subsampling design $\zeta^*$ is concentrated on, 
at most, $p = q + 1$ intervals, and its density $f_{\zeta^*}(z)$  
is either equal to the density $f_Z(z)$ of the covariate $Z_i$
(on $\mathcal{Z}^* = \sigma \XX^* + \mu$) or it is equal to $0$ (elsewhere)
such that, also here, the optimal subsampling can be implemented quite easily.

A further reduction of the optimization problem can be achieved
by utilizing symmetry properties.
Therefore, we consider the transformation of sign change, $g(x) = -x$,
and assume that the distribution of the covariate is symmetric,
$f_X(-x) = f_X(x)$ for all $x$.
For a continuous design $\xi$, the design $\xi^g$ transformed by
sign change has density $f_{\xi^g}(x) = f_{\xi}(-x)$ and, 
thus, satisfies the boundedness condition $f_{\xi^g}(x) \leq f_X(x)$,
when the distribution of $X_i$ is symmetric,
and has the same value for the $D$-criterion as $\xi$,
$\log(\det(\M(\xi^g))) = \log(\det(\M(\xi)))$.
By the concavity of the $D$-criterion,
standard invariance arguments can be used 
as in Pukelsheim (\citeyear[Chapter~13]{pukelsheim1993optimal}) 
and \cite{heiligers1992invariant}.
In particular, any continuous design $\xi$ is dominated
by its symmetrization $\bar\xi = (\xi + \xi^g)/2$ 
with density $f_{\bar\xi}(x) = (f_\xi(x) + f_\xi(-x))/2 \leq f_X(x)$
such that $\log(\det(\M(\bar\xi))) \geq \log(\det(\M(\xi)))$
\cite[Chapter 13.4]{pukelsheim1993optimal}.
Hence, we can restrict the search for a $D$-optimal subsampling design
to symmetric designs $\bar\xi$ with density 
$f_{\bar\xi}(-x) = f_{\bar\xi}(x)$ which are invariant with respect to
sign change (${\bar\xi}^g = \bar\xi$).
For these symmetric subsampling designs $\bar\xi$,
the moments $m_{k}(\bar\xi)$ are zero for odd $k$
and positive when $k$ is even.
Hence, the information matrix $\M(\bar\xi)$ is an even checkerboard matrix 
\cite[see][]{Jones_2018} with positive entries $m_{j + j^\prime}(\bar\xi)$
for even index sums and entries equal to zero when the index sum is odd.
The inverse $\M(\bar\xi)^{-1}$ of the information matrix $\M(\bar\xi)$ 
shares the structure of an even checkerboard matrix.
Thus, the sensitivity function $\psi(x,\bar\xi)$ is a polynomial 
with only terms of even order and is, hence, a symmetric function
of $x$. This leads to a simplification of the representation of the 
optimal subsampling design in Theorem~\ref*{th:opt-design-degree-q}
because the support $\XX^*$ of the optimal subsampling design $\xi^*$ 
will be symmetric, too.

\begin{corollary}
	\label{cor:opt-design-degree-q-symm}
	In polynomial regression of degree $ q $ with a
	symmetrically distributed covariate $ X_{i} $ with density $f_X(x)$, 
	the $D$-optimal subsampling design $\xi^*$ 
	with density 
	$f_{\xi^*}(x) = \sum_{k=0}^r f_X(x) \1_{\mathcal{I}_k}(x) $
	has symmetric boundaries $a_1, \dots, a_{2r}$
	of the intervals $\mathcal{I}_0 = [a_{1}, \infty)]$, 
	$\mathcal{I}_k = [a_{2k + 1}, a_{2k}]$, and $\mathcal{I}_r = (-\infty, a_{2r}]$, 
	i.\,e.\ $a_{2r + 1 - k} = -a_k$ and,
	accordingly, $\mathcal{I}_{r  - k} = -\mathcal{I}_k$.
\end{corollary}

This characterization of the optimal subsampling design $\xi^*$
will be illustrated in the next two sections 
for ordinary linear regression ($q = 1$)
and for quadratic regression ($q = 2$).

\section{Optimal Subsampling for Linear Regression}
\label{sec:linear}
In the case of ordinary linear regression
$Y_{i} =\beta_0 + \beta_1 X_{i} + \varepsilon_{i}$, 
we have 
\begin{align*}
	\M(\xi) = \begin{pmatrix}
		\alpha     & m_1(\xi)      \\
		m_1(\xi)          & m_2(\xi)     
	\end{pmatrix},
\end{align*}
for the information matrix of any subsampling design $\xi$.
The inverse $\M(\xi)^{-1}$ of the information matrix is given by
\begin{align*}
	\M(\xi)^{-1} = \frac{1}{\alpha m_2(\xi) - m_1(\xi)^2}
	\begin{pmatrix}
		m_2(\xi)     & -m_1(\xi)      \\
		-m_1(\xi)          & \alpha     
	\end{pmatrix},
\end{align*}
and the sensitivity function 
\begin{align}
	\label{eq:sensitivity-linear}
	\psi(x, \xi) = \frac{1}{\alpha m_2(\xi) - m_1(\xi)^2}
	(m_2(\xi) - 2 m_1(\xi) x + \alpha x^2)     
\end{align}
is a polynomial of degree two in $x$.
The $D$-optimal continuous subsampling design $\xi^*$ has density 
$f_\xi(x) = f_X(x)$ for $x \leq a_2$ and for $x \geq a_1$
while $f_\xi(x) = 0$ for $a_2 < x < a_1$.
The corresponding subsampling design then accepts those units $i$
for which $x_i \leq a_2$ or $x_i \geq a_1$,
and rejects all units $i$ for which $a_2 < x_i < a_1$.

To obtain the $D$-optimal continuous subsampling design $\xi^*$
by Theorem~\ref{th:opt-design-degree-q},
the boundary points $a_1$ and $a_2$ have to be determined 
to solve the two non-linear equations
\begin{equation}
	\label{eq:lin-cond1}
	\Prob(X_i \leq a_2) + \Prob(X_i \geq a_1) = \alpha
\end{equation}
and 
\begin{equation*}
	\label{eq:lin-cond2}
	\psi(a_2, \xi^*) = \psi(a_1, \xi^*) \, .
\end{equation*}
By equation~\eqref{eq:sensitivity-linear},
the latter condition can be written as
\begin{equation*}
	\label{eq:lin-cond2a}
	\alpha a_2^2 - 2 m_1(\xi^*) a_2 = \alpha a_1^2 - 2 m_1(\xi^*) a_1 \, ,
\end{equation*}
which can be reformulated as
\begin{equation}
	\label{eq:lin-cond2b}
	\alpha (a_1 + a_2) = 2 m_1(\xi^*) \, .
\end{equation}

When the distribution of $X_i$ is symmetric,
Corollary~\ref{cor:opt-design-degree-q-symm}
provides symmetry $a_2 = -a_1$ of the boundary points.
This is in agreement with condition~\eqref{eq:lin-cond2b}
because $m_1(\xi^*) = 0$ in the case of symmetry.
Further, by the symmetry of the distribution, 
$\Prob(X_i \leq a_2) = \Prob(X_i \geq a_1) = \alpha/2$,
and $a_1$ has to be chosen as the $(1 - \alpha/2)$-quantile
of the distribution of $X_i$ to obtain the $D$-optimal continuous subsampling design.

\begin{example}[normal distribution]
	\label{ex:lin-normal}
	If the covariate $X_i$ comes from a standard normal distribution,
	then the optimal boundaries are the $(\alpha/2)$-
	and the $(1 - \alpha/2)$-quantile $\pm z_{1 - \alpha/2}$,
	and unit $i$ is accepted when $|x_i| \geq z_{1 - \alpha/2}$.
	\par
	For $X_i$ having a general normal distribution 
	with mean $\mu$ and variance $\sigma^2$,
	the optimal boundaries remain to be the $(\alpha/2)$-
	and $(1 - \alpha/2)$-quantile $a_2 = \mu - \sigma z_{1 - \alpha/2}$
	and $a_1 = \mu + \sigma z_{1 - \alpha/2}$, respectively, by 
	Theorem~\ref{th:equivariant}.
\end{example}

This approach applies accordingly to all distributions which are obtained
by a location or scale transformation of a symmetric distribution:
units will be accepted if their values of the covariate
lie in the lower or upper $(\alpha/2)$-tail of the distribution.
This procedure can be interpreted as a theoretical counterpart 
in one dimension of the IBOSS method 
proposed by \cite{wang2019information}.

However, for an asymmetric distribution of the covariate $X_i$, 
the optimal proportions for
sampling from the upper and lower tail may differ.
By condition~\eqref{eq:cond1}, there will be a proportion
$\alpha_1$, $0 \leq \alpha_1 \leq \alpha$,
for the upper tail and $\alpha_2 = \alpha - \alpha_1$
for the lower tail such that $a_1$ is the $(1 - \alpha_1)$-quantile 
and $a_2$ is the $\alpha_2$-quantile of the
distribution of the covariate $X_i$, respectively.
In view of condition~\eqref{eq:lin-cond2b},
neither $\alpha_1$ nor $\alpha_2$ can be zero.
Hence, the optimal subsampling design will have
positive, but not necessarily equal mass at both tails.
This will be illustrated in the next example.

\begin{example}[exponential distribution]
	\label{ex:lin-exponential}
	If the covariate $X_i$ comes from a standard exponential distribution
	with density $ f_{X}(x) = e^{-x},\; x \ge 0 $, we conclude from Theorem~\ref{th:opt-design-degree-q} that 
	$ f_{\xi^{*}}(x) = f_{X}(x) \1_{[0,b]\cup[a,\infty)}(x) $ 
	with $a = a_1$ and $b = a_2$ when $a_2 \geq 0$.
	Otherwise, when $a_2 < 0$, the density $f_X(x)$ of the covariate $X_i$
	vanishes on the left interval $\mathcal{I}_1 = (-\infty, a_2]$
	because the support of the distribution of $X_i$ 
	does not cover the whole range of $\R$.
	In that case, we may formally let $b = 0$.
	Then, we can calculate the entries of $ \M(\xi^{*}) $ 
	as functions of $ a $ and $ b $ as
	\begin{align*}
		m_{1}(\xi^{*}) 
		&= 1 + (a + 1)e^{-a} - (b + 1)e^{-b} \\
		m_{2}(\xi^{*}) 
		&= 2 + (a^2 + 2a + 2)e^{-a} - (b^2 + 2b + 2)e^{-b} \,.
	\end{align*}
	To obtain the optimal solutions for $a$ and $b$ in the case $a_2 \geq 0$, 
	the two non-linear equations \eqref{eq:lin-cond1} and \eqref{eq:lin-cond2b} 
	have to be satisfied which become here  
	$e^{-b} - e^{-a} = 1 - \alpha$
	and $\alpha (a + b) = 2 m_1(\xi^{*})$.
	
	If $a_2 < 0$ would hold, the first condition reveals $a = - \log(\alpha)$ 
	and, hence, $m_1(\xi^{*}) = \alpha (a + 1)$.
	There, similar to the proof of Theorem~\ref{Theo:normal} below,
	the second condition has to be relaxed to $\psi(a,\xi^{*}) \geq \psi(0,\xi^{*})$
	which can be reformulated to 
	$\alpha a \geq 2 m_1(\xi^{*}) = 2\alpha (a + 1)$
	and yields a contradiction.
	Thus, this case can be excluded, 
	and $a_2$ has to be larger than $0$ for all $\alpha$.
	
	For selected values of $ \alpha $, numerical results are presented
	in Table~\ref{Table:Exp}. 
	Additionally to the optimal values for $a$ and $b$, 
	also the proportions 
	$\Prob(X_i \leq b)$
	and
	$\Prob(X_i \geq a)$
	are presented in Table~\ref{Table:Exp}
	together with the percentage of mass
	allocated to the left interval $[0, b]$.
	In Figure~\ref{Figure:Exp}, 
	the density $f_{\xi^{*}}$ of the optimal subsampling design $\xi^*$
	and the corresponding sensitivity function $\psi(x,\xi^*)$ 
	are exhibited for $\alpha = 0.5$ and $\alpha = 0.3$. 
	Vertical lines indicate the positions of the boundary points
	$a$ and $b$, 
	and the dotted horizontal line displays the threshold $s^*$.
	\begin{table}[h]
		\begin{center}
			\caption{Numeric values for the boundary points $a$ and $b$ 
				for selected values of the subsampling proportion $\alpha$ 
				in the case of standard exponential $X_i$}
			\begin{tabular}{l|ccccc} \toprule
				{$\alpha$} & {$b$} & {$\Prob(X_i \leq b)$} & {$a$} & {$\Prob(X_i \geq a)$} 
					& {\% of mass on $[0,b]$} \\ \midrule
				0.5  & 0.39572 & 0.32681 & 1.75335 & 0.17319 & 65.36 \\ 
				0.3  & 0.21398 & 0.19264 & 2.23153 & 0.10736 & 64.21 \\
				0.1  & 0.06343 & 0.06146 & 3.25596 & 0.03854 & 61.46 \\
				0.01 & 0.00579 & 0.00577 & 5.46588 & 0.00423 & 57.71 \\ \bottomrule
			\end{tabular}
			\label{Table:Exp}
		\end{center}
	\end{table}
	\begin{figure}[htb]
		\centering
		\begin{subfigure}[t]{.475\textwidth}
			\centering
			\includegraphics[width=\linewidth]{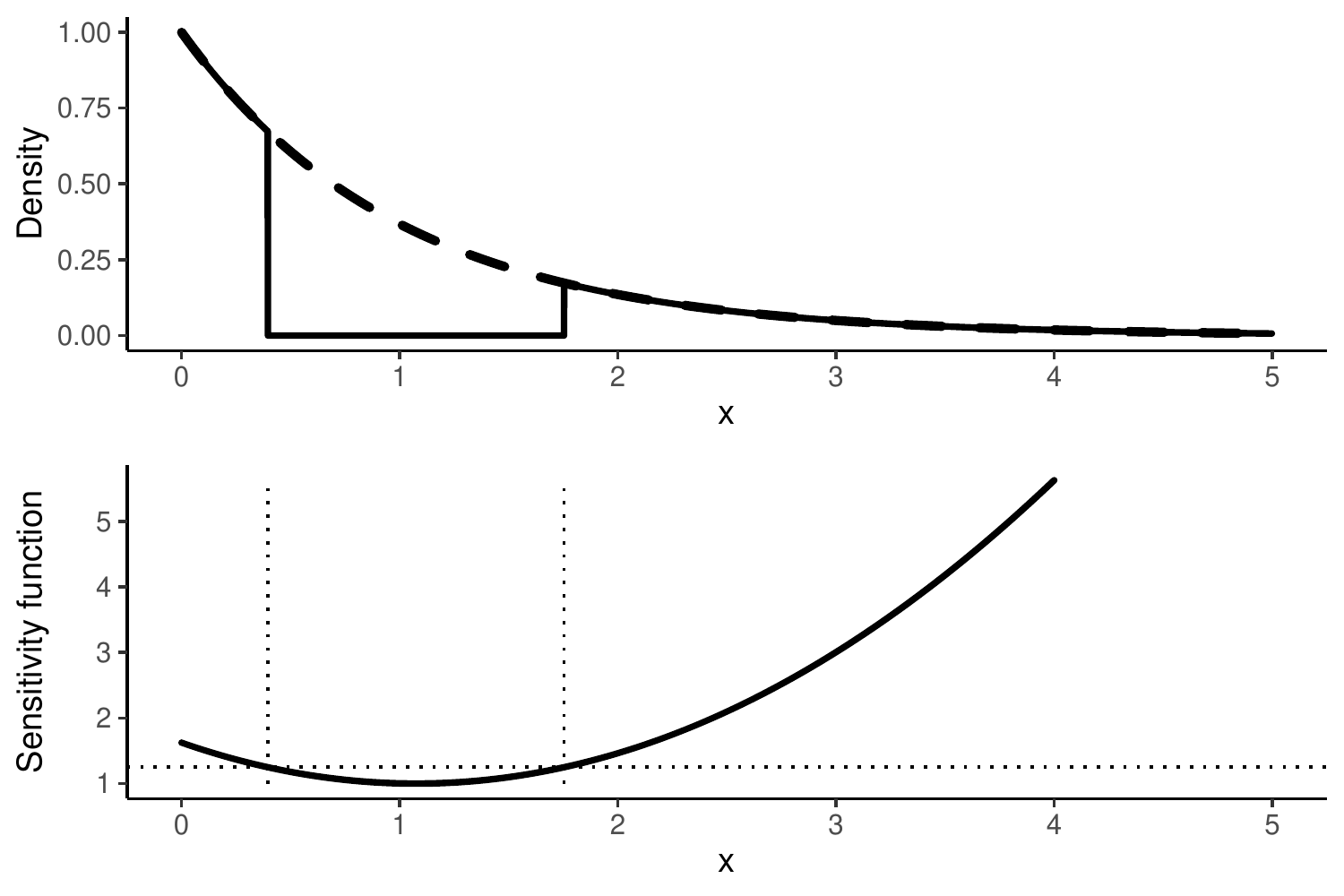}
			\caption{$\alpha = 0.5$}
		\end{subfigure}
		\hfill
		\begin{subfigure}[t]{.475\textwidth}
			\centering
			\includegraphics[width=\linewidth]{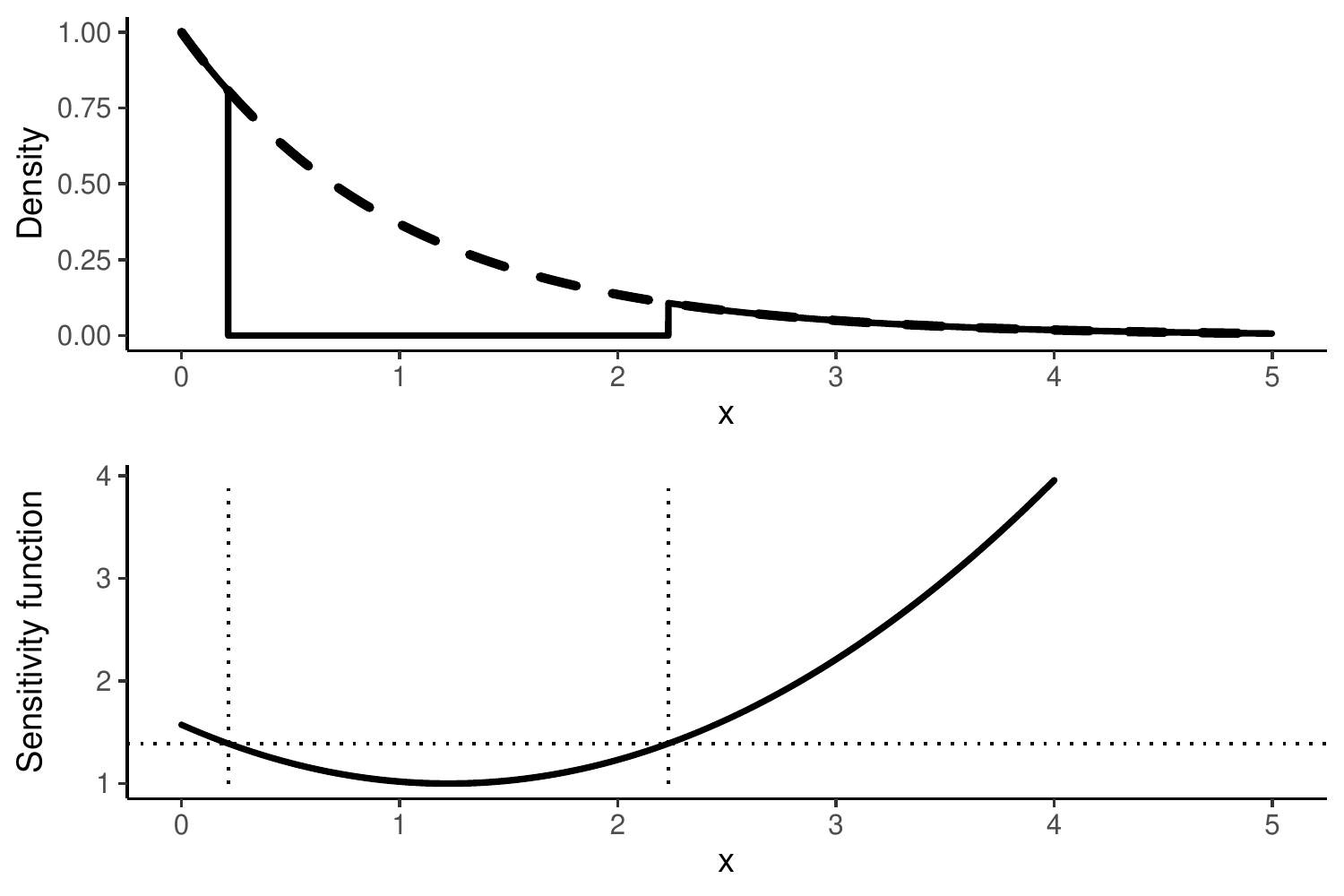}
			\caption{$\alpha = 0.3$.}
		\end{subfigure}
		\caption{Density of the optimal subsampling design (solid line)
			and the standard exponential distribution (dashed line, upper panels),
			and sensitivity functions (lower panels)
			for subsampling proportions $\alpha = 0.5$ (left) and $\alpha = 0.3$ (right)}
		\label{Figure:Exp}
	\end{figure}
	As could have been expected, 
	less mass is assigned to the right tail 
	of the right-skewed distribution
	because observations from the right tail are more influential
	and, thus, more observations seem to be required
	on the lighter left tail for compensation.
	\par
	For $X_i$ having an exponential distribution 
	with general intensity $\lambda > 0 $ (scale $1/\lambda$),
	the optimal boundary points remain to be the 
	same quantiles as in the standard exponential case, 
	$a_1 = a/\lambda$
	and $a_2 = b/\lambda$ associated with the proportion $\alpha$, by 
	Theorem~\ref{th:equivariant}.
\end{example}

\section{Optimal Subsampling for Quadratic Regression}
\label{sec:quadratic}

In the case of quadratic regression
$Y_{i} =\beta_0 + \beta_1 X_{i} + \beta_2 X_{i}^2 + \varepsilon_{i}$ 
we have 
	\begin{align}
		\label{eq:info-quadratic}
	\M(\bar{\xi}) = \begin{pmatrix}
		\alpha     & 0      & m_2(\bar{\xi}) \\
		0          & m_2(\bar{\xi})    & 0   \\
		m_2(\bar{\xi})       & 0      & m_4(\bar{\xi})
	\end{pmatrix},
\end{align}
for the information matrix of a symmetric subsampling design $\bar{\xi}$.
The inverse $\M(\bar{\xi})^{-1}$ of the information matrix is given by
\begin{align*}
	\M(\bar{\xi})^{-1} = \frac{1}{\alpha m_4(\bar{\xi}) - m_2(\bar{\xi})^2}
	\begin{pmatrix}
		m_4(\bar{\xi})     & 0      & -m_2(\bar{\xi}) \\
		0          & \alpha \frac{m_4(\bar{\xi})}{m_2(\bar{\xi})} - m_2(\bar{\xi})    & 0   \\
		-m_2(\bar{\xi})       & 0      & \alpha
	\end{pmatrix},
\end{align*}
and the sensitivity function 
\begin{align}
	\label{eq:sensitivity-quadratic}
	\psi(x, \bar{\xi}) = \frac{1}{\alpha m_4(\bar{\xi}) - m_2(\bar{\xi})^2}
	(m_4(\bar{\xi}) - 3 m_2(\bar{\xi}) x^2 
		+ \alpha \frac{m_4(\bar{\xi})}{m_2(\bar{\xi})} x^2
		+ \alpha x^4)     
\end{align}
is a polynomial of degree four and is symmetric in $x$.

According to Corollary~\ref{cor:opt-design-degree-q-symm},
the density $ f_{\xi^{*}}(x) $ of the $ D $-optimal continuous subsampling design $\xi^*$
has, at most, three intervals that are symmetrically placed around zero, 
where the density is equal to the bounding density $  f_{X}(x) $,
and  $ f_{\xi^{*}}(x) $ is equal to zero elsewhere. 
Thus the density $ f_{\xi^{*}}(x) $ of the $D$-optimal subsampling design
has the shape
\begin{equation}
	\label{eq:density-quadratic}
	f_{\xi^{*}}(x) =  f_{X}(x) \1_{(-\infty,-a]\cup[-b,b]\cup[a,\infty)}(x) \, ,
\end{equation}
where $a > b \geq 0$. We formally allow $b = 0$
which means that $\psi(0,\xi^*) \leq s^{*} = \psi(a,\xi^*)$
and that the density $f_{\xi^*}(x)$ is concentrated
on only two intervals, 
$f_{\xi^{*}}(x) =  f_{X}(x) \1_{(-\infty,-a]\cup[a,\infty)}(x)$.
Although the information matrix will be non-singular 
even in the case of two intervals ($b=0)$,
the optimal subsampling design will include a non-degenerate 
interior interval $[-b, b]$ in many cases, $b > 0$, as illustrated below
in Examples~\ref{ex:quad-normal} and \ref{ex:quad-uniform}. 
However, for a heavy-tailed distribution of the covariate $X_i$,
the interior interval may vanish in the optimal subsampling design
as shown in Example~\ref{ex:quad-tdist}.

To obtain the $D$-optimal continuous subsampling design $\xi^*$
by Corollary~\ref{cor:opt-design-degree-q-symm},
the boundary points $a = a_1$ and $b = a_2 \geq 0$ 
have to be determined 
to solve the two non-linear equations
\begin{equation}
	\label{eq:cond1}
	\Prob(|X_i| \leq b) + \Prob(|X_i| \geq a) = \alpha
\end{equation}
and 
\begin{equation}
	\label{eq:cond2}
	\psi(b, \xi^*) = \psi(a, \xi^*) \, .
\end{equation}
By equation~\eqref{eq:sensitivity-quadratic},
the latter condition can be written as
\begin{equation*}
	\label{eq:quad-cond2a}
	\alpha m_2(\xi^*) b^4 + \left(\alpha m_4(\xi^*)  - 3 m_2(\xi^*)^2\right) b^2
	= \alpha m_2(\xi^*) a^4 + \left(\alpha m_4(\xi^*)  - 3 m_2(\xi^*)^2\right) a^2 \, ,
\end{equation*}
which can be reformulated as
\begin{equation}
	\label{eq:quad-cond2b}
	\alpha m_2(\xi^*) (a^2 + b^2) = \alpha m_4(\xi^*)  - 3 m_2(\xi^*)^2 \, .
\end{equation}

For finding the optimal solution,
we use the Newton method implemented 
in the \textbf{\textsf{R}} package \textit{nleqslv} by \cite{nleqslv} 
to calculate numeric values for $a$ and $b$ based on equations 
\eqref{eq:cond1} and \eqref{eq:cond2}
for various symmetric distributions.

The case $b = 0$ relates to the situation of only two intervals ($r = 1 < q$).
There, condition~\eqref{eq:cond1} simplifies to
$a = q_{1 - \alpha/2}$, where $q_{1 - \alpha/2}$ is the 
$(1 - \alpha/2)$-quantile of the distribution of the covariate $X_i$,
and equation~\eqref{eq:cond2} 
has to be relaxed to $\psi(0, \xi^*) \leq \psi(a, \xi^*)$,
similar to the case $b = 0$ in Example~\ref{ex:lin-exponential}. 

\begin{example}[normal distribution]
	\label{ex:quad-normal}
	For the case that the covariate $X_i$ 
	comes from a standard normal distribution,
	results are given in Table~\ref{Table:Normal} for selected 
	values of $\alpha$. 
	\begin{table}[h]
		\begin{center}
			\caption{Numeric values for the boundary points $a$ and $b$ 
				for selected values of the subsampling proportion $\alpha$ 
				in the case of standard normal $X_i$}
			\begin{tabular}{l|ccccc} 
				\toprule
				{$\alpha$} & {$a$} & {$1 - \Phi(a)$} & {$b$} & {$2\Phi(b) - 1$} 
				& {\% of mass on $[-b,b]$} 
				\\ 
				\midrule
				0.5  & 1.02800 & 0.15198 & 0.24824 & 0.19605 & 39.21 
				\\ 
				0.3  & 1.34789 & 0.08885 & 0.15389 & 0.12231 & 40.77  
				\\
				0.1  & 1.88422 & 0.02977 & 0.05073 & 0.04046 & 40.46 
				\\
				0.01 & 2.73996 & 0.00307 & 0.00483 & 0.00386 & 38.55 
				\\ 
				\bottomrule
			\end{tabular}
			\label{Table:Normal}
		\end{center}
	\end{table}
	Additionally to the optimal values for $a$ and $b$, 
	also the proportions 
	$\Prob(X_i \geq a) = \Prob(X_i \leq -a) = 1 - \Phi(a)$
	and
	$\Prob(-b  \leq X_i \leq b) = 2\Phi(b) - 1$
	are presented in Table~\ref{Table:Normal}
	together with the percentage of mass $(2\Phi(b) - 1) / \alpha$
	allocated to the interior interval $[-b, b]$.
	In Figure~\ref{Figure:Normal}, 
	the density $f_{\xi^{*}}$ of the optimal subsampling design $\xi^*$
	and the corresponding sensitivity function $\psi(x,\xi^*)$ 
	are exhibited for $\alpha = 0.5$ and $\alpha = 0.1$.
	\begin{figure}[htb]
		\centering
		\begin{subfigure}[t]{.475\textwidth}
			\centering
			\includegraphics[width=\linewidth]{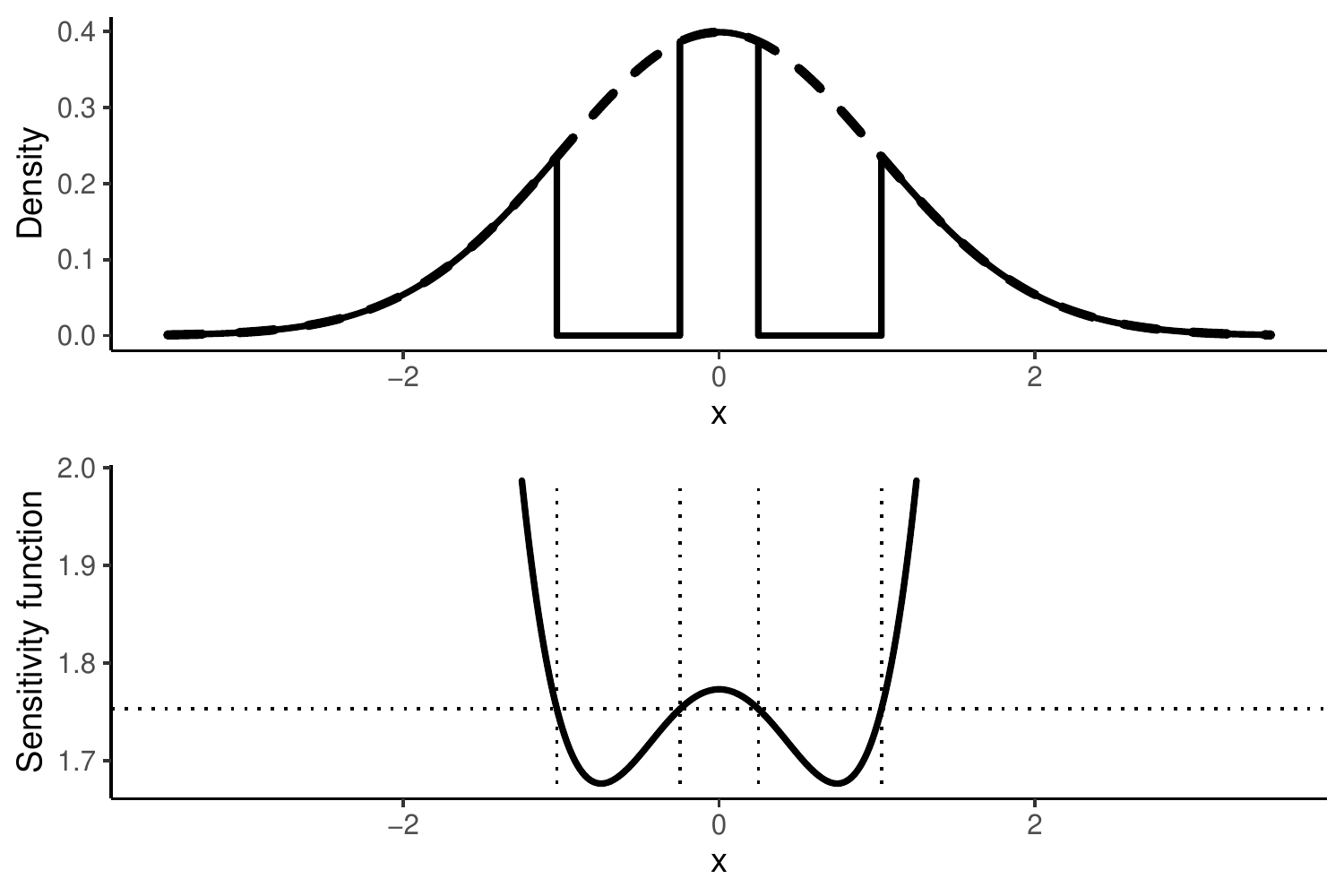}
			\caption{$\alpha = 0.5$}
		\end{subfigure}
		\hfill
		\begin{subfigure}[t]{.475\textwidth}
			\centering
			\includegraphics[width=\linewidth]{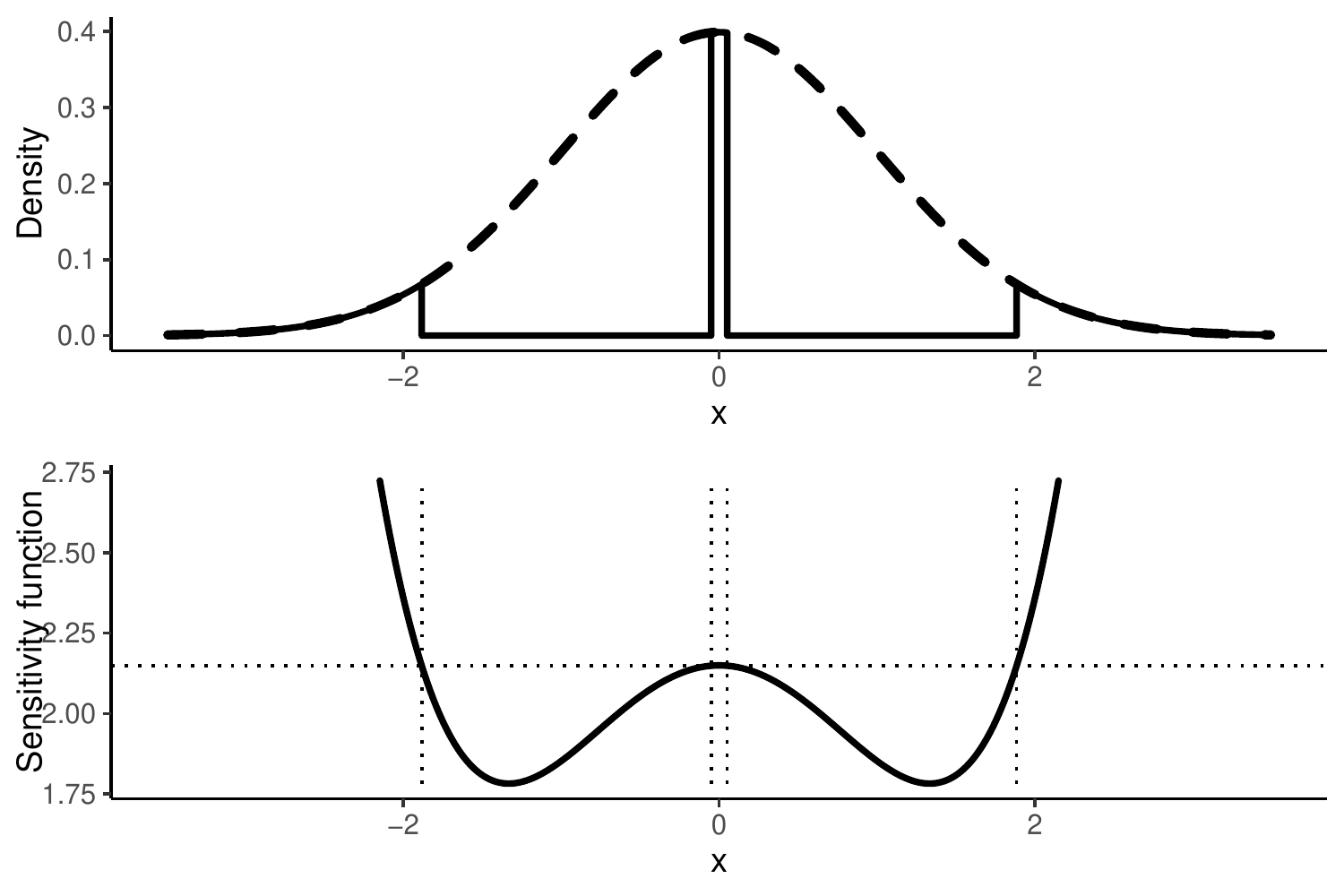}
			\caption{$\alpha = 0.1$}
		\end{subfigure}
		\caption{Density of the optimal subsampling design (solid line)
			and the standard normal distribution (dashed line, upper panels),
			and sensitivity functions (lower panels)
			for subsampling proportions $\alpha = 0.5$ (left) and $\alpha = 0.1$ (right)}
		\label{Figure:Normal}
	\end{figure}
	Vertical lines indicate the positions of the boundary points
	$-a$, $-b$, $b$, and $a$, respectively.	
	In the subplots of the sensitivity function, the dotted horizontal line 
	displays the threshold $s^*$.
	For other values of $\alpha$, the plots are looking similar.
	\par
	The numerical results in Table~\ref{Table:Normal} suggest 
	that the interior interval $[-b, b]$ does not vanish 
	for any $\alpha$ ($0 < \alpha < 1$). 
	This will be established in the following theorem.
	\begin{theorem}
		\label{Theo:normal}
		In quadratic regression with standard normal 
		covariate $X_i$,
		for any subsampling proportion $ \alpha \in (0,1) $, 
		the $D$-optimal subsampling design $\xi^{*}$ 
		has density
		$f_{\xi^{*}}(x) =  f_{X}(x) \1_{(-\infty, -a]\cup[-b, b]\cup[a, \infty)}(x)$
		with $a > b > 0$.
	\end{theorem}
	\par
	For $X_i$ having a general normal distribution 
	with mean $\mu$ and variance $\sigma^2$,
	the optimal boundary points remain to be the 
	same quantiles as in the standard normal case, 
	$a_1, a_4 = \mu \pm \sigma a$
	and $a_2, a_3 = \mu \pm \sigma b$, by 
	Theorem~\ref{th:equivariant}.
\end{example}

\begin{example}[uniform distribution]
	\label{ex:quad-uniform}
	If the covariate $X_i$ is uniformly distributed on $[-1,1]$ 
	with density $ f_{X}(x) = \frac{1}{2}\1_{[-1,1]}(x)$,
	we can obtain analytical results for the dependence 
	of the subsampling design on the proportion $ \alpha $ to be selected. 

	The distribution of $X_i$ is symmetric.
	By Corollary~\ref{cor:opt-design-degree-q-symm},
	the density of the $D$-optimal continuous subsampling design $\xi^*$ 
	has the shape
	\begin{equation}
		\label{eq:dens-unif}
		f_{\xi^{*}}(x) = \frac{1}{2} \1_{[-1,-a] \cup [-b,b] \cup [a,1]}(x) \, ,
	\end{equation}
	where we formally allow $a = 1$ or $b = 0$
	resulting in only one or two intervals of support.
	The relevant entries in the information matrix $\M(\xi^{*})$
	are $m_2(\xi^*) = \frac{1}{3}(1 - a^3 + b^3)$ 
	and $m_4(\xi^*) = \frac{1}{5}(1 - a^5 + b^5)$.
	If, in Corollary~\ref{cor:opt-design-degree-q-symm}, 
	the boundary points $a_1$ and $a_2$ satisfy 
	$a_1 \leq 1$ and $a_2 \geq 0$, then
	$a = a_1$ and $b = a_2$ are the solution of the two equations 
	$a - b = 1 - \alpha$
	and
	$\alpha m_2(\xi^*) (a^2 + b^2) = \alpha m_4(\xi^*)  - 3 m_2(\xi^*)^2$
	arising from conditions~\eqref{eq:cond1} and \eqref{eq:quad-cond2b}.
	On the other hand, if there exist solutions $a$ and $b$
	of these equations such that $0 < b <  a < 1$,
	then these are the boundary points in the 
	representation~\eqref{eq:dens-unif},
	and the density of the optimal subsampling design is
	supported by three proper intervals.
	Solving the two equations results in
	\begin{align}
		\label{eq:unifa}
		a(\alpha) = \frac{1}{2} & (1 - \alpha) + \Biggl( \frac{1}{180(1 - \alpha)} 
		\Bigl(45 - 15\alpha + 15\alpha^2 -45\alpha^3 + 20\alpha^4 \notag
		\\
		& \quad \mbox{} - 4\alpha \sqrt{5} 
		\sqrt{45 - 90\alpha + 90\alpha^2 - 75\alpha^3 + 57\alpha^4 - 27\alpha^5 + 5\alpha^6} \, \Bigr) \Biggr)^{1/2} 
	\end{align}
	and
	\begin{equation}
		\label{eq:unifb}
		b(\alpha) = a(\alpha) - (1 - \alpha)
	\end{equation}
	for the dependence of $a$ and $b$ on $\alpha$.
	The values of $a$ and $b$ are plotted in Figure~\ref{Figure:a-and-b}. 
	\begin{figure}[h]
		\centering
		\includegraphics[width=.5\textwidth]{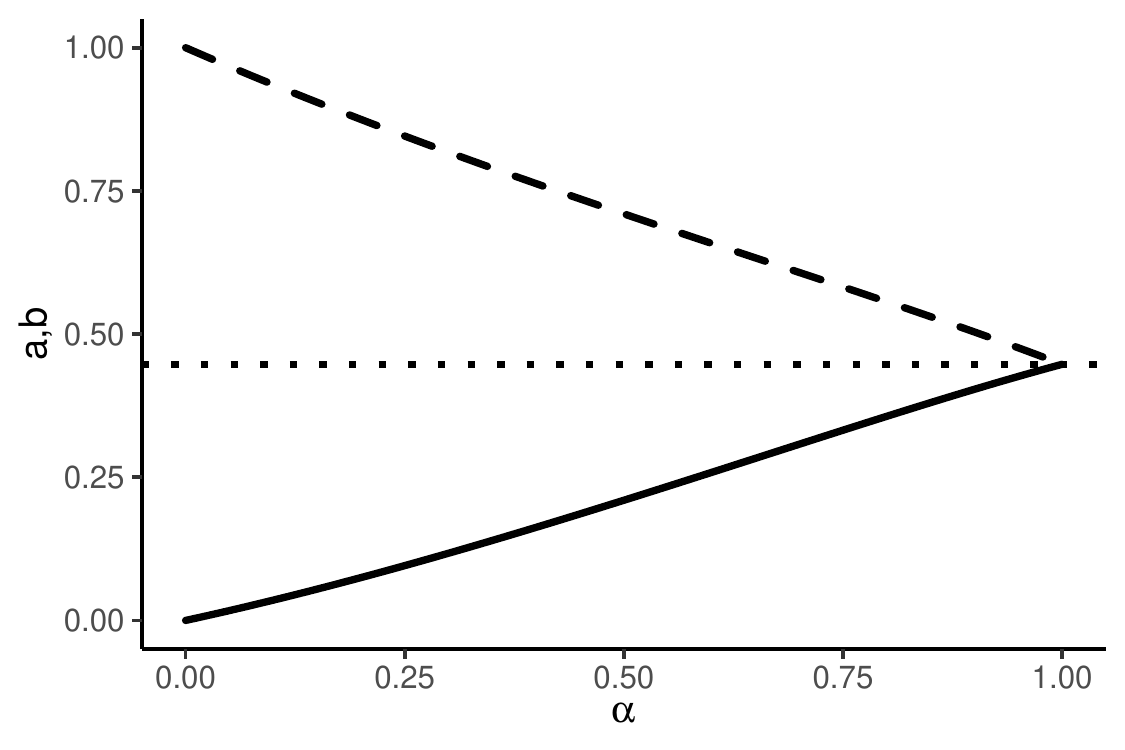}
		\caption{Boundary points $a$ (dashed) and $b$ (solid) of the $ D $-optimal subsampling design in the case of uniform $ X_{i} $ on $ [-1,1] $ as functions of $\alpha$}
		\label{Figure:a-and-b}
	\end{figure}
	There it can be seen that $0 < a < b < 1$ for all $\alpha$ and 
	that $a$ and $b$ both tend to $1/\sqrt{5}$ 
	as $\alpha$ tends to $1$.
	Similar to the case of the normal distribution, 
	the resulting values and illustrations 
	are given in Table~\ref{Table:Unif} and Figure~\ref{Figure:Unif}. 
	Note that the mass of the interior interval 
	$ \Prob(-b \leq X_i \leq b) $ 
	is equal to $ b $ itself as $ X_{i} $ is uniformly distributed on $ [-1,1] $.
	\begin{table}[h]
		\begin{center}
			\caption{Values for the boundary points $a$ and $b$ 
				for selected values of the subsampling proportion $\alpha$ 
				in the case of uniform $X_i$ on $[-1, 1]$}
			\begin{tabular}{l|cccc} \toprule
				{$\alpha$} & {$a$} & {$\Prob(X_i \geq a)$} & {$b = \Prob(-b \leq X_i \leq b)$} 
					& {\% of mass on $[-b,b]$} \\ \midrule
				0.5  & 0.70983 & 0.14508 & 0.20983 & 41.97 \\ 
				0.3  & 0.81737 & 0.09132 & 0.11737 & 39.12  \\ 
				0.1  & 0.93546 & 0.03227 & 0.03546 & 35.46 \\ 
				0.01 & 0.99336 & 0.00332 & 0.00336 & 33.55 \\ \bottomrule 
			\end{tabular}
			\label{Table:Unif}
		\end{center}
	\end{table}
	\begin{figure}[htb]
		\centering
		\begin{subfigure}[t]{.475\textwidth}
			\centering
			\includegraphics[width=\linewidth]{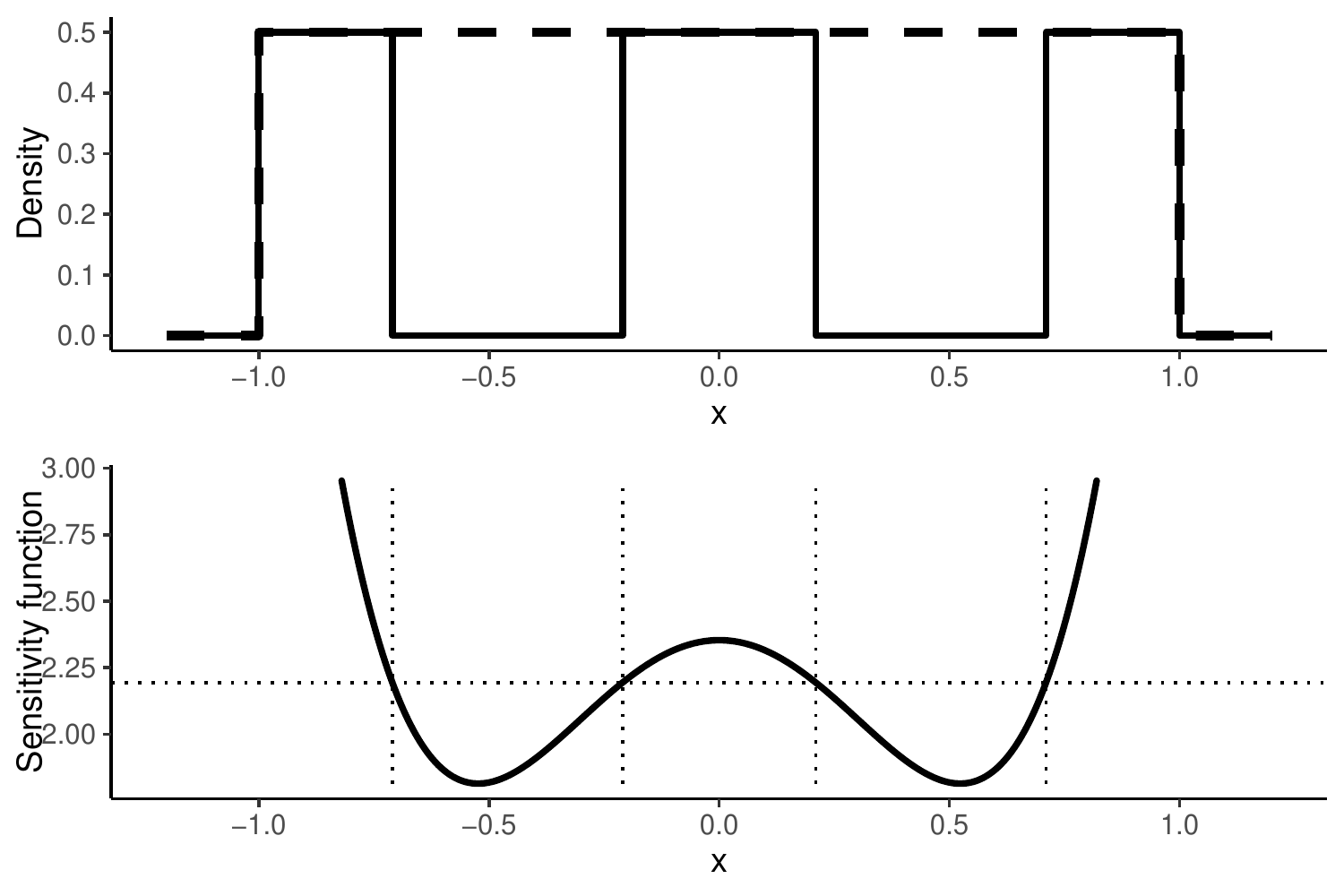}
			\caption{$\alpha = 0.5$}
		\end{subfigure}
		\hfill
		\begin{subfigure}[t]{.475\textwidth}
			\centering
			\includegraphics[width=\linewidth]{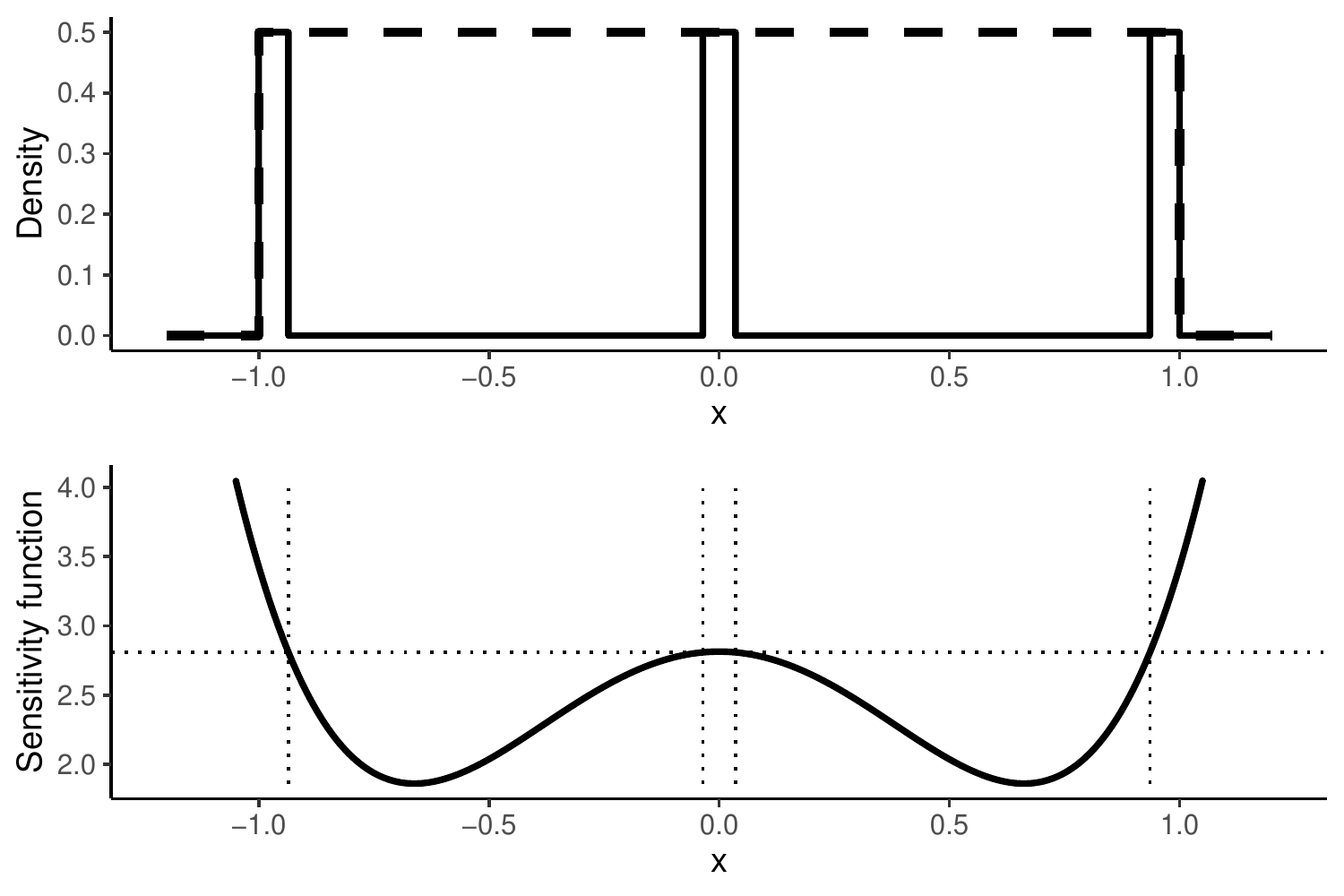}
			\caption{$\alpha = 0.1$}
		\end{subfigure}
		\caption{Density of the optimal subsampling design (solid line)
		and the uniform distribution on $[-1,1]$ (dashed line, upper panels),
		and sensitivity functions (lower panels)
		for subsampling proportions $\alpha = 0.5$ (left) and $\alpha = 0.1$ (right)}
		\label{Figure:Unif}
	\end{figure}
	Also here, in Figure~\ref{Figure:Unif}, 
	vertical lines indicate the positions of the boundary points
	$-a$, $-b$, $b$, and $a$, 
	and the dotted horizontal line displays the threshold $s^*$.
	Moreover, the percentage of mass at the different intervals is
	displayed in Figure~\ref{Figure:Percentage-a-and-b}.
	\begin{figure}[htb]
		\centering
		\begin{subfigure}[t]{.475\textwidth}
			\centering
			\includegraphics[width=\linewidth]{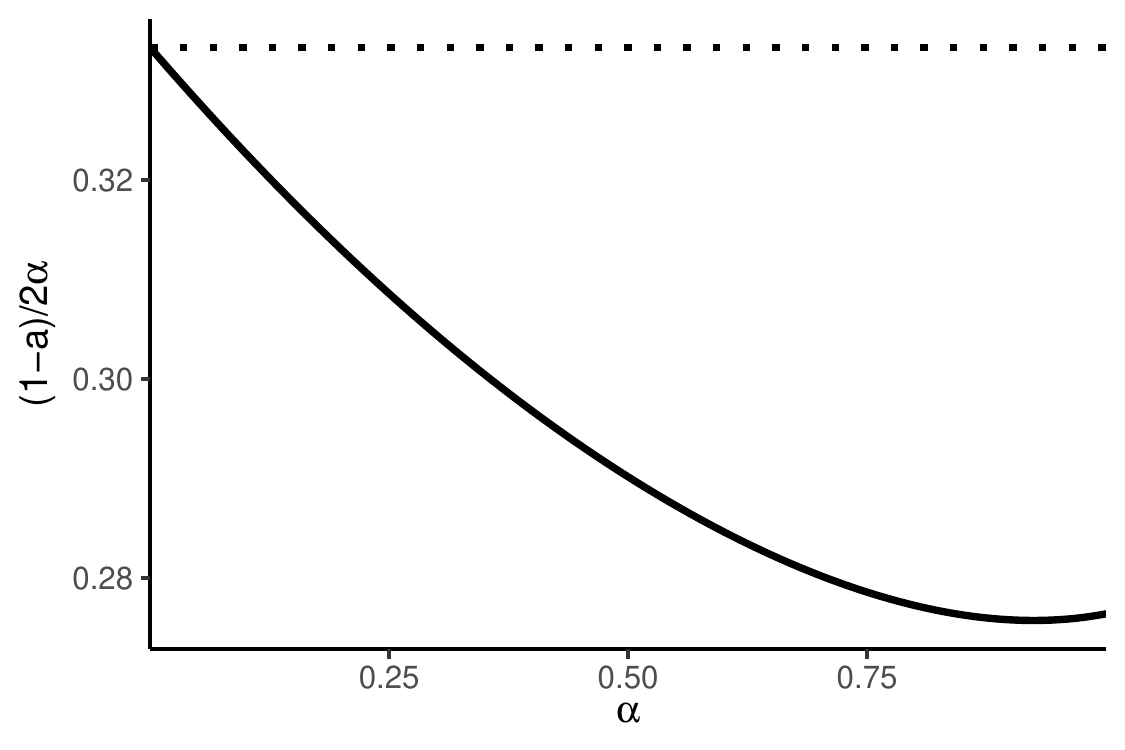}
		\end{subfigure}
		\hfill
		\begin{subfigure}[t]{.475\textwidth}
			\centering
			\includegraphics[width=\linewidth]{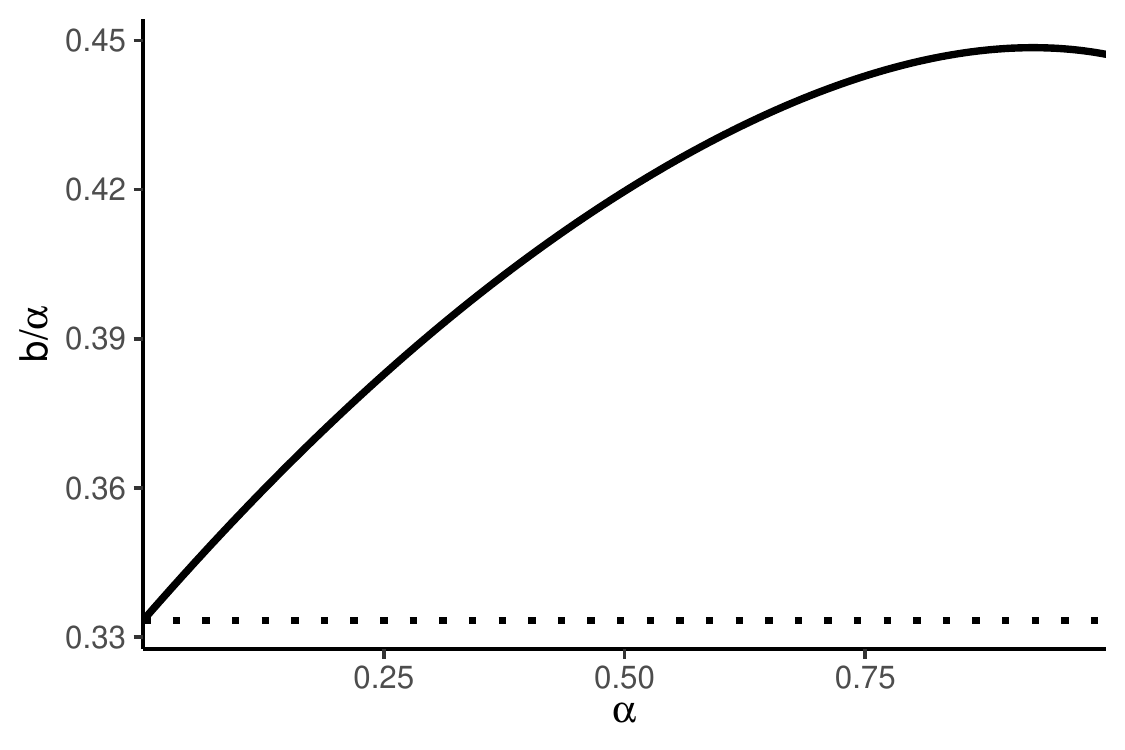}
		\end{subfigure}
		\caption{Percentage of mass on the support intervals $[a,1]$ (left) and $[-b,b]$ (right) of the $ D $-optimal subsampling design in the case of uniform $ X_{i} $ on $ [-1,1] $ as a function of $\alpha$}
		\label{Figure:Percentage-a-and-b}
	\end{figure}
	
	The results in Table~\ref{Table:Unif} and Figure~\ref{Figure:Percentage-a-and-b} suggest 
	that the percentage of mass on all three intervals 
	$[-1,-a]$, $[-b,b]$, and $[a,1]$ tend to $1/3$ as $\alpha$ tends to $0$. 
	We establish this in the following theorem. 
	\begin{theorem}
		\label{th:unif-alpha-to-0}
		In quadratic regression with  
		covariate $X_i$ uniformly distributed on $[-1,1]$, let 
		$ \xi_{\alpha}^{*} $ be the optimal subsampling design for subsampling proportion $\alpha$, 
		$0 < \alpha < 1$, defined in equations~\eqref{eq:unifa} and \eqref{eq:unifb}. 
		Then $ \lim_{\alpha \to 0} \xi_{\alpha}^{*}([-b,b]) / \alpha = 1/3 $.
	\end{theorem}
	It is worth-while mentioning that the percentages of mass 
	displayed in Figure~\ref{Figure:Percentage-a-and-b} are not monotonic over the whole
	range of $ \alpha \in (0,1) $, as, for example the percentage of mass at the interior 
	interval $[-b, b]$ is increasing from $0.419666$ at $b = 0.50$ to
	$0.448549$ at $b = 0.92$ and then slightly decreasing back again to 
	$0.447553$ at $b = 0.99$.
	
	Finally, it can be checked that, for all $\alpha$, the solutions
	satisfy $0 < b < a < 1$ such that the optimal subsampling designs
	are supported on three proper intervals.
\end{example}

In the two preceding examples it could be noticed
that the mass of observations is of comparable size
for the three supporting intervals in the case 
of a normal and of a uniform distribution
with light tails.
This may be different in the case of a 
heavy-tailed distribution
for the covariate $ X_{i} $ as the $t$-distribution.
\begin{example}[$ t $-distribution]
	\label{ex:quad-tdist}
	For the case that the covariate $X_i$ 
	comes from a $ t $-distribution with $ \nu $ degrees of freedom, 
	we observe a behavior which differs substantially from the normal case 
	of Example~\ref{ex:quad-normal}.
	The interior interval typically has less mass than the outer intervals 
	and may vanish for some values of $ \alpha $.
	We show this in the case of the least possible number $ \nu = 5 $ 
	of degrees of freedom to maintain an existing fourth moment,
	which appears in the information matrix of the 
	$ D $-optimal continuous subsampling design $ \xi^{*} $
	while maximizing the dispersion.

	\begin{theorem}
		\label{Theo:t5}
		In quadratic regression with $t$-distributed covariate
		$X_i \sim t_{5}$ with five degrees of freedom, 
		there is a critical value 
		$\alpha^* \approx 0.082065$ of the subsampling proportion $ \alpha $ such that
		the $D$-optimal subsampling design $\xi^{*}$ has  
		\begin{itemize}
			\item[(i)] 
				density
				$f_{\xi^{*}}(x) =  f_{X}(x) \1_{(-\infty,-a]\cup[-b,b]\cup[a,\infty)}(x)$
				with $ a > b > 0 $ for $ \alpha < \alpha^* $. 
			\item[(ii)] 
				density
				$f_{\xi^{*}}(x) 
					=  f_{X}(x) \1_{(-\infty,-t_{5,1-\alpha/2}]\cup[t_{5,1-\alpha/2},\infty)}(x)$,
				where $ t_{5,1-\alpha/2} $ is the $(1 - \alpha/2)$-quantile of the 
				$ t_5 $-distribution, for $ \alpha \geq \alpha^* $. 
		\end{itemize}
	\end{theorem}  

	For illustration, numerical results are given in Table~\ref{Table:t5}. 
	The percentage of mass on the interior interval $ [-b,b] $ is equal to zero for all larger values of $ \alpha $
	as stated in Theorem~\ref{Theo:t5}. 
	The percentage of mass on $ [-b,b] $ decreases with increasing subsampling proportion $ \alpha $ 
	before vanishing entirely.
	\begin{table}[h]
		\begin{center}
			\caption{Values for the boundary points $a$ and $b$ 
					for selected values of the subsampling proportion $\alpha$ 
					in the case of $ t_{5} $-distributed $X_i$}
			\begin{tabular}{l|ccccc} \toprule
				{$\alpha$} & {$a$} & {$\Prob(X_i \geq a)$} & {$b$} & {$\Prob(-b \leq X_i \leq b)$} 
				& {\% of mass on $[-b,b]$} \\ \midrule
				0.10  & 2.01505 & 0.05000 & 0\phantom{.00202} & 0\phantom{.00153} 
					& \phantom{1}0\phantom{.03} \\ 
				0.07  & 2.31512 & 0.03423 & 0.00202 & 0.00153 & \phantom{1}2.03  \\ 
				0.03  & 3.09141 & 0.01356 & 0.00380 & 0.00288 & \phantom{1}4.74 \\ 
				0.01  & 4.18942 & 0.00429 & 0.00187 & 0.00142 & 14.23 \\ \bottomrule 
			\end{tabular}
			\label{Table:t5}
		\end{center}
	\end{table}

Further calculations provide that the critical value $ \alpha^{*} $, 
where a the $ D $-optimal subsampling design switches from a 
three-interval support to a two-interval support, increases with the number of degrees $ \nu $ of freedom of the 
$ t $-distribution and converges to one when $ \nu $ tends to infinity.
This is in accordance with the results for the normal distribution in Example~\ref{ex:quad-normal} as the $ t $-distribution converges in distribution to a standard normal distribution for $\nu \to \infty$.
We have given numeric values for the crossover points for selected degrees of freedom in Table~\ref{Table:t5-crossover},
where $\nu = \infty$ relates to the normal distribution.
The corresponding value $\alpha^* = 1$ indicates 
that the $D$-optimal subsampling design is supported 
by three intervals for all $\alpha$ in this case.
	\begin{table}[h]
		\begin{center}
			\caption{Values of the critical value $ \alpha^{*} $ for selected degrees of 
			freedom $ \nu $ of the $ t $-distribution}
			\begin{tabular}{l|ccccccc} \toprule
				{$\nu$} & {$5$} & {$6$} & {$7$} & {$8$}	& {$9$} & {$30$} & {$\infty$} 
				\\ 
				\midrule
				$\alpha^{*}$  & 0.08207 & 0.34670 & 0.50374 & 0.60125 & 0.66670 & 0.92583 & 1
				\\ 
				\bottomrule 
			\end{tabular}
			\label{Table:t5-crossover}
		\end{center}
	\end{table}
\end{example}

\section{Efficiency}
\label{sec:efficiency}
To exhibit the gain in using a $D$-optimal subsampling design
compared to random subsampling, we consider the 
performance of
the uniform random subsampling design $\xi_{\alpha}$ of size $\alpha$,
which has density $f_{\xi_{\alpha}}(x) = \alpha f_X(x)$,
compared to the $D$-optimal subsampling design 
$ \xi_{\alpha}^{*} $ with mass $ \alpha$.

More precisely, the $D$-efficiency  
of any subsampling design $ \xi $ with mass $\alpha$ is defined as
\begin{equation*}
	\label{eq:efficiency}
	\eff_{D,\alpha}(\xi) = \left(\frac{\det(\M(\xi))}{\det(\M(\xi_{\alpha}^{*}))}\right)^{1/p},
\end{equation*}
where $ p $ is the dimension of the parameter vector $\B$. 
For this definition the homogeneous version $(\det(\M(\xi)))^{1/p}$
of the $D$-criterion is used which satisfies the homogeneity condition
 $(\det(\lambda \M(\xi)))^{1/p} = \lambda (\det(\M(\xi)))^{1/p}$
 for all $\lambda > 0$
 \cite[see][Chapter~6.2]{pukelsheim1993optimal}.

For uniform random subsampling, 
the information matrix is given by
$ \M(\xi_{\alpha}) = \alpha \M(\xi_{1}) $,
where $ \M(\xi_{1}) $ is the information matrix for the full sample
with raw moments $m_k(\xi_{1}) = \E (\X_{i}^{k})$
as entries in the $(j,j^\prime)$th position, $j + j^\prime - 2 = k$.
Thus, the $D$-efficiency $\eff_{D,\alpha}(\xi_{\alpha})$ of uniform random subsampling
can be nicely interpreted: 
the sample size (mass) required to obtain the same precision (in terms of the $D$-criterion), 
as when the $D$-optimal subsampling design $\xi_{\alpha}^*$ of mass $\alpha$ is used,
is equal to the inverse of the efficiency $\eff_{D,\alpha}(\xi_{\alpha})^{-1}$ times $\alpha$.
For example, if the efficiency $\eff_{D,\alpha}(\xi_{\alpha})$ is equal
to $0.5$, then twice as many observations would be needed under 
uniform random sampling 
than for a $D$-optimal subsampling design of size $\alpha$.
Of course, the full sample has higher information than any proper subsample
such that, obviously, for uniform random subsampling, 
$\eff_{D,\alpha}(\xi_{\alpha}) \geq \alpha$ holds for all $\alpha$.

For the examples of Sections~\ref{sec:linear} and \ref{sec:quadratic},
the efficiency of uniform random subsampling is given 
 in Table~\ref{Table:Eff} for selected values of $\alpha$
\begin{table}[h]
	\begin{center}
		\caption{Efficiency of uniform subsampling w.r.t. $ D $-optimality for selected values of the subsampling proportion $ \alpha $}
		\begin{tabular}{l|l|cccc} \toprule
			\multicolumn{2}{c|}{ } & \multicolumn{4}{c}{$\alpha$} \\
			\multicolumn{2}{c|}{ } & 0.5 & 0.3 & 0.1 & 0.01 \\
			\midrule
			linear regression & normal  & 0.73376 & 0.61886 & 0.47712 & 0.34403 \\
			 & exponential     & 0.73552 & 0.61907 & 0.46559 & 0.30690 \\
			\midrule
			quadratic regression & normal & 0.73047 & 0.59839 & 0.41991 & 0.24837 \\
			 & uniform   & 0.78803 & 0.70475 & 0.62411 & 0.58871 \\
			 &{$ t_{5} $} & {0.66400} & {0.50656} & {0.29886} & {0.10941} \\
			 &{$ t_{9} $} & {0.70390} & {0.56087} & {0.36344} & {0.17097} \\
			\bottomrule
		\end{tabular}
		\label{Table:Eff}
	\end{center}
\end{table}
and exhibited in Figure~\ref{Figure:Efficiencies}
for the full range of $\alpha$ between $0$ and $1$
(solid lines).
\begin{figure}[htb]
	\centering
	\begin{subfigure}[t]{.475\textwidth}
		\centering
		\includegraphics[width=\linewidth]{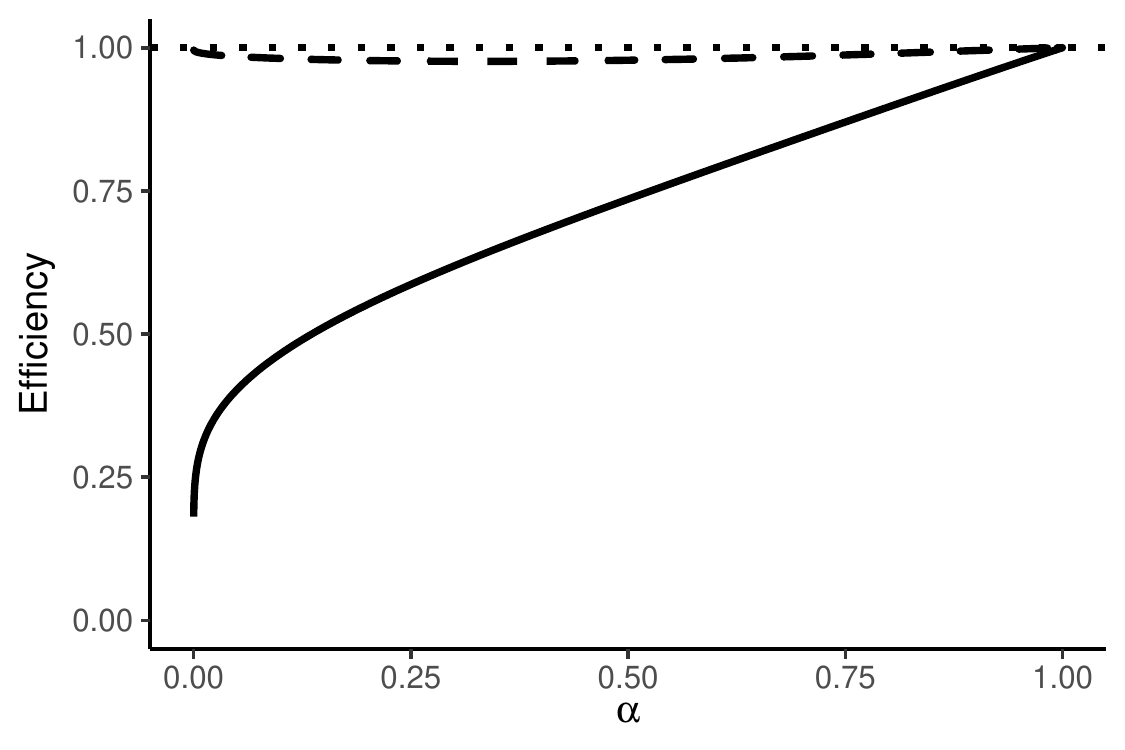}
		\caption{Linear regression, exponential covariate}
	\end{subfigure}
	\hfill
	\begin{subfigure}[t]{.475\textwidth}
		\centering
		\includegraphics[width=\linewidth]{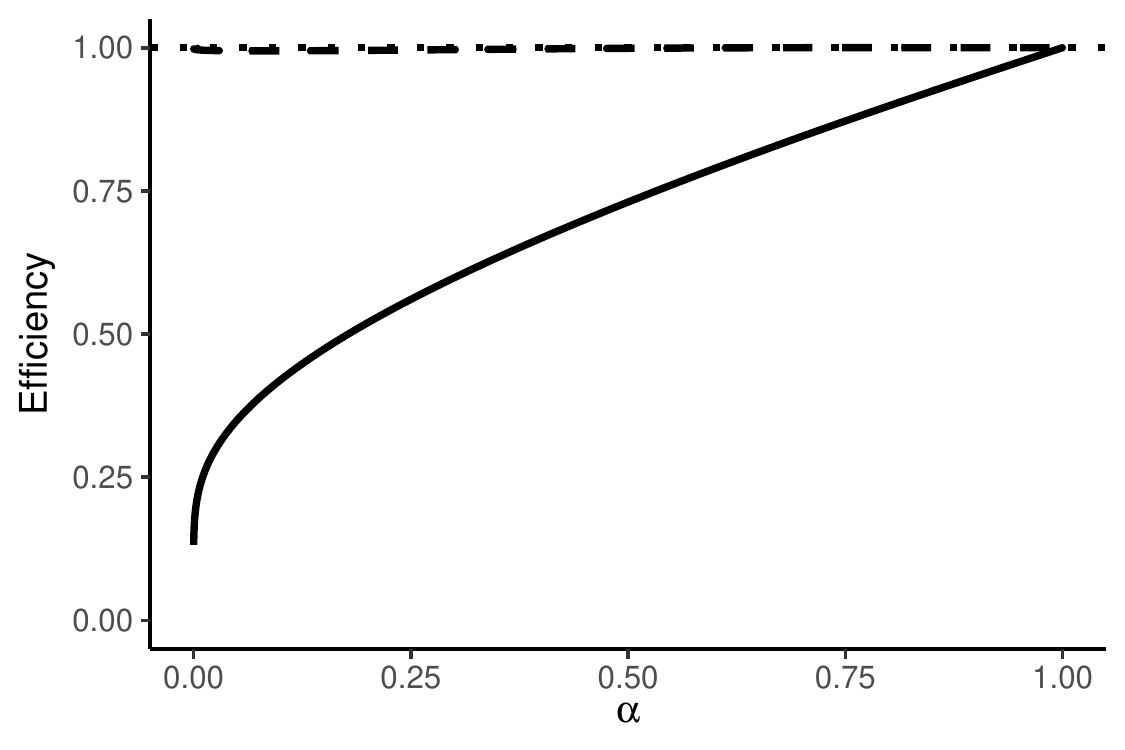}
		\caption{Quadratic regression, normal covariate}
	\end{subfigure}
	\medskip
	\begin{subfigure}[t]{.475\textwidth}
		\centering
		\includegraphics[width=\linewidth]{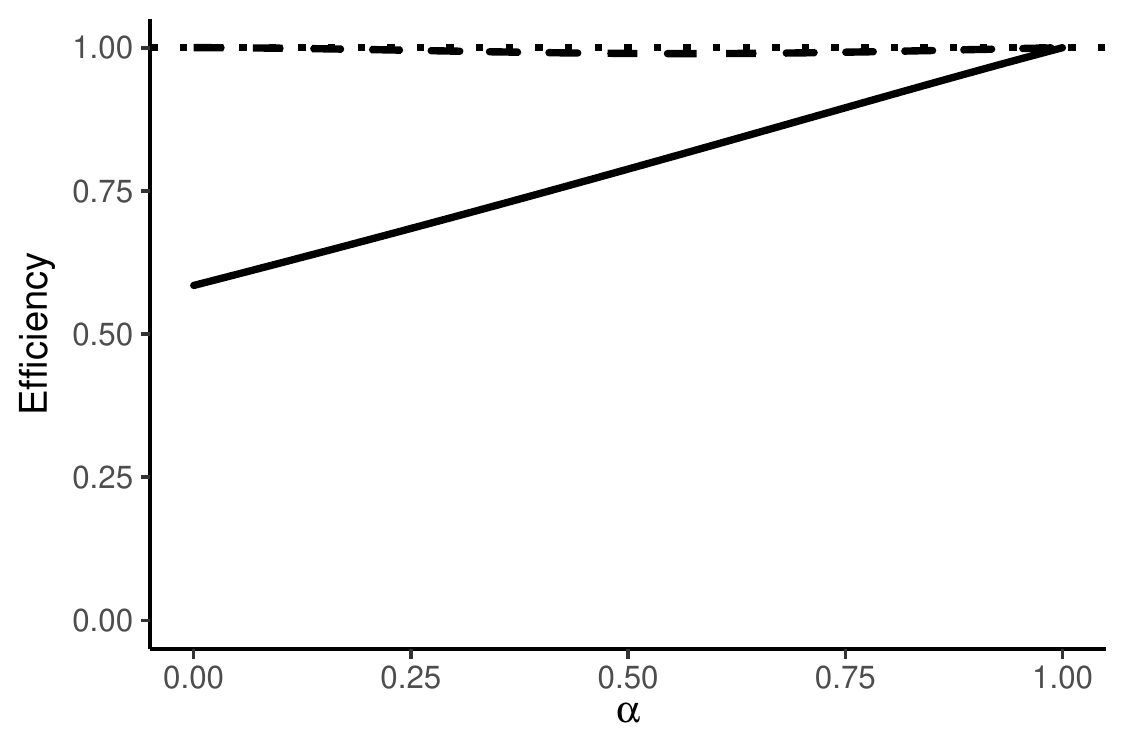}
		\caption{Quadratic regression, uniform covariate}
	\end{subfigure}
	\hfill
	\begin{subfigure}[t]{.475\textwidth}					
		\centering
		\includegraphics[width=\linewidth]{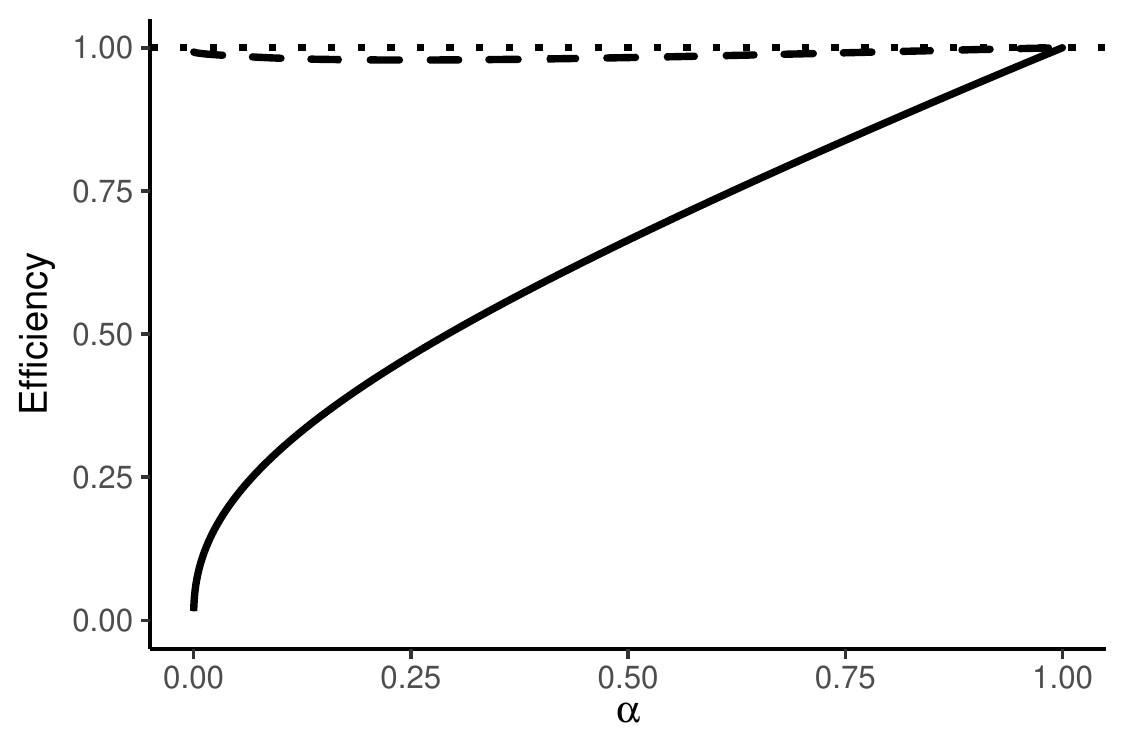}
		\caption{Quadratic regression, $ t_{5} $ covariate}
	\end{subfigure}
	\caption{Efficiency of uniform random subsampling (solid line)
		and of an IBOSS-type subsampling design (dashed line) w.r.t. $ D $-optimality}
	\label{Figure:Efficiencies}
\end{figure}
Here the determinant of the information matrix 
is determined as in the examples 
of Sections~\ref{sec:linear} and \ref{sec:quadratic}
for the optimal subsampling designs $\xi_{\alpha}^{*}$
either numerically or by explicit formulas where available.

Both Table~\ref{Table:Eff} and Figure~\ref{Figure:Efficiencies}
indicate that the efficiency of uniform random subsampling
is decreasing in all cases
when the proportion $\alpha$ of subsampling gets smaller.
In the case of quadratic regression with uniformly distributed covariate,
the decrease is more or less linear with a minimum value 
of approximately $0.58$ 
when $\alpha$ is small.
In the other cases, where the distribution of the covariate 
is unbounded, the efficiency apparently decreases faster,
when the proportion $\alpha$ is smaller than $10\%$,
and tends to $0$ for $\alpha \to 0$.

The latter property can be easily seen for linear regression
and symmetric distributions: there, the efficiency
$\eff_{D,\alpha}(\xi_{\alpha})$ of uniform random sampling
is bounded from above by $c / q_{1 - \alpha/2}$,
where $c = \E(X_i^2)^{1/2}$ is a constant
and $q_{1 - \alpha/2}$ is the $(1 - \alpha/2)$-quantile
of the distribution of the covariate.
When the distribution is unbounded like the normal distribution,
then these quantiles tend to infinity for $\alpha \to 0$
and, hence, the efficiency tends to $0$.
Similar results hold for quadratic regression and
asymmetric distributions.

In any case, as can be seen from Table~\ref{Table:Eff},
the efficiency of uniform random subsampling is quite low
for reasonable proportions $\alpha \leq 0.1$
and, hence, the gain in using the $D$-optimal subsampling
design is substantial.

By equivariance arguments as indicated above in the examples of
Sections~\ref{sec:linear} and \ref{sec:quadratic}, the present efficiency
considerations carry over directly to a covariate having a general normal,
exponential, or uniform distribution, respectively.

In the IBOSS approach by \cite{wang2019information}, 
  of the proportion $\alpha$ is taken from both tails of 
the data. The corresponding continuous subsampling 
design $\xi_{\alpha}^{\prime}$
would be to have two intervals $(-\infty, b]$ and $[a, \infty)$ 
and to choose the boundary points $a$ and $b$ 
to be the $(1 - \alpha/2)$- and $(\alpha/2)$-quantile of the
distribution of the covariate, respectively. For linear
regression, it can been seen from 
Corollary~\ref{cor:opt-design-degree-q-symm}
that the subsampling design $\xi_{\alpha}^{\prime}$ is $D$-optimal
when the distribution of the covariate is symmetric.
As the IBOSS procedure does not use
prior knowledge of the distribution, it would be
tempting to investigate the efficiency of the 
corresponding continuous subsampling design 
$\xi_{\alpha}^{\prime}$ under asymmetric distributions. 
For the exponential distribution, this efficiency
$\eff_{D,\alpha}(\xi_{\alpha}^{\prime})$ is added to the
upper left panel in Figure~\ref{Figure:Efficiencies}
by a dashed line. There the subsampling design $\xi_{\alpha}^{\prime}$ 
shows a remarkably high efficiency over the whole
range of $\alpha$ with a minimum value  
$0.976$ at $ \alpha = 0.332 $. 

As an extension of IBOSS for quadratic regression,
we may propose a procedure which takes proportions 
$\alpha/3$ from both tails of the data as well as from 
the center of the data. 
This procedure can be performed without any prior 
knowledge of the distribution of the covariate.
The choice of the proportions $\alpha/3$
is motivated by the standard case $D$-optimal design
on an interval where one third of the weight is allocated
to each of the endpoints and to the midpoint of the interval, respectively.
For a symmetric distribution, the corresponding
continuous subsampling design 
$\xi_{\alpha}^{\prime\prime}$ can be defined by
the boundary points $a$ and $b$ to be
the $(1 - \alpha/3)$- and $(1/2 + \alpha/6)$-quantile of the
distribution of the covariate, respectively. 
In the case of the uniform distribution, the subsampling design
$\xi_{\alpha}^{\prime\prime}$ is the limiting $D$-optimal subsampling
design for $\alpha \to 0$ by Theorem~\ref{th:unif-alpha-to-0}. 
In Figure~\ref{Figure:Efficiencies},
the efficiency $\eff_{D,\alpha}(\xi_{\alpha}^{\prime\prime})$ 
is shown by dashed lines for the whole range of $\alpha$ 
for the uniform distribution as well as for the normal and for the t- distribution
in the case of quadratic regression.
In all three cases, the subsampling design 
$\xi_{\alpha}^{\prime\prime}$ is highly efficient over the whole
range of $\alpha$ with minimum values 
$0.994$ at $ \alpha = 0.079 $ for the normal distribution, 
$0.989$ at $ \alpha = 0.565 $ for the uniform distribution, 
and
$0.978$ at $ \alpha = 0.245 $ for the $ t_{5} $-distribution,
respectively. This is of particular interest for the $ t_{5} $-distribution,
where the interior interval of the $ D $-optimal subsampling design 
$ \xi_{\alpha}^{*} $ is considerably smaller than of the IBOSS-like 
subsampling design $ \xi_{\alpha}^{\prime\prime} $ and even vanishes entirely for 
$ \alpha > \alpha^* \approx 0.08 $.
However, we only tested this extension of IBOSS for quadratic regression 
for symmetric distributions of the covariate. Further investigations for non-symmetric
distributions is necessary.

\section{Concluding Remarks}
\label{sec:discussion}

In this paper we have considered a theoretical approach 
to evaluate subsampling designs under distributional assumptions
on the covariate in the case of polynomial regression on a
single explanatory variable.
We first reformulated the constrained equivalence theorem
under Kuhn-Tucker conditions in
\cite{sahm2001note} to characterize
the $D$-optimal continuous subsampling design for general distributions of the covariate.
For symmetric distributions of the covariate we concluded the following. 
The $ D $-optimal subsampling design is equal to the bounding distribution 
in its support and the support of the optimal subsampling design will be the union 
of at most $ q + 1 $ intervals that are symmetrically placed around zero. 
Further we have found that in the case of quadratic regression
the $ D $-optimal subsampling design has three support intervals with
positive mass for all $ \alpha \in (0,1) $, whereas the interior interval
vanishes for some $ \alpha $ for a $ t $-distributed covariate.
In contrast to that, for linear regression,
always two intervals are required at the tails of the distribution.

The main emphasis in this work was on $D$-optimal subsampling designs. 
But many of the results may be extended to other
optimality criteria like $A$- and $E$-optimality from the Kiefer's 
$\Phi_q$-class of optimality criteria, $IMSE$-optimality for predicting
the mean response, or optimality criteria based on subsets or linear 
functionals of parameters.

The $D$-optimal subsampling designs show a high performance compared to
uniform random subsampling. In particular, for small proportions,
the efficiency of uniform random subsampling tends to 
zero when the distribution of the covariate is unbounded.
This property is in accordance with the observation 
that estimation based on subsampling according
to IBOSS is ``consistent'' in the sense that the mean squared error 
goes to zero with increasing population size even when the size of
the subsample is fixed.

We propose a generalization of the IBOSS method to quadratic 
regression which does not require prior knowledge of the
distribution of the covariate and which performs remarkably well
compared to the optimal subsampling design. However, an extension to 
higher order polynomials does not seem to be obvious.

\appendix
\section{Proofs}

Before proving Theorem~\ref{th:opt-design-degree-q},
we establish two preparatory lemmas on properties of the 
sensitivity function $\psi(x, \xi)$ for a continuous subsampling design $\xi$
with density $f_{\xi}(x)$
and reformulate an equivalence theorem on constraint design 
optimality by \cite{sahm2001note} for the present setting.
The first lemma deals with the shape of the sensitivity function.

\begin{lemma}
	\label{lem:sens-polynomial-2q}
	The sensitivity function $\psi(x, \xi)$ is a polynomial of degree $2q$
	with positive leading term.
\end{lemma}
\begin{proof}[Proof of Lemma~\ref{lem:sens-polynomial-2q}]
	For a continuous subsampling design $\xi$ with density $f_\xi(x)$, 
	the information matrix $\M(\xi) $ and, hence, its inverse 
	$\M(\xi)^{-1} $ is positive definite.
	Thus the last diagonal element $m^{(pp)}$ of $\M(\xi)^{-1} $ 
	is positive and, as $\f(x) = (1, x, \dots, x^q)^\top$,
	the sensitivity function $ \psi(x, \xi) = \f(x)^{\top} \M(\xi)^{-1} \f(x) $
	is a polynomial of degree $2q$ with coefficient $m^{(pp)} > 0$ 
	of the leading term.
\end{proof}

The second lemma reveals a distributional property of the
sensitivity function considered as a function in the covariate $X_i$.
	
\begin{lemma}
	\label{lem:sens-continuous}
	The random variable $\psi(X_i, \xi)$ 
	has a continuous cumulative distribution function.
\end{lemma}
\begin{proof}[Proof of Lemma~\ref{lem:sens-continuous}]
	As the sensitivity function $\psi(x, \xi)$ is a non-constant 
	polynomial by Lemma~\ref{lem:sens-polynomial-2q},
	the equation $\psi(x, \xi) = s$ has only finitely many roots
	$x_1, \dots, x_{\ell}$, $\ell \leq 2q$, say, by the fundamental theorem of 
	algebra. Hence, 
	$\Prob(\psi(X_i, \xi) = s) = \sum_{k=1}^{\ell} \Prob(X_i = x_k) = 0$
	by the continuity of the distribution of $X_i$
	which proves the continuity of the cumulative distribution 
	function of $\psi(X_i, \xi)$.
	\hfill $\Box$
\end{proof}

With the continuity of the distribution of $\psi(X_i,\xi^*)$
the following equivalence theorem can be obtained from
Corollary~1(c) in \cite{sahm2001note} for the present setting
by transition from the directional derivative to the sensitivity 
function and considering $\R$ as the design region.

\begin{theorem}[Equivalence Theorem]
	\label{th:equivalence}
	The subsampling design $\xi^*$ is $D$-optimal if and only if
	there exist a threshold $s^*$ 
	and a subset $\XX^*$ of $\R$ such that
	\begin{itemize}
		\item[(i)]	
		the $D$-optimal subsampling design $\xi^{*}$ is given by
		\begin{equation*}
			\label{eq:opt-dens-general}
			f_{\xi^{*}}(x) = f_{X}(x) \1_{\XX^*}(x) \,
		\end{equation*}
		\item[(ii)] $\psi(x, \xi^*) \geq s^*$ for $x \in \XX^*$, and
		\item[(iii)] $\psi(x, \xi^*) < s^*$ for $x \not\in \XX^*$.
	\end{itemize}	
\end{theorem}

As $\Prob(\psi(X_i, \xi^*) \geq s^*) = \Prob(X_i \in \XX^*) 
= \int f_{\xi^*}(x) \diff x = \alpha$, the threshold $s^*$
is the $(1 - \alpha)$-quantile of the distribution of $\psi(X_i, \xi^*)$.

\begin{proof}[Proof of Theorem~\ref{th:opt-design-degree-q}]
By Lemma~\ref{lem:sens-polynomial-2q} 
the sensitivity function $\psi(x, \xi)$ is a polynomial in $x$ 
of degree $2q$ with positive leading term.
Using the same argument as in the proof of 
Lemma~\ref{lem:sens-continuous} we obtain
that there are at most $2q$ roots of the equation 
$\psi(x, \xi^*) = s^*$ and, hence, there are at most $2q$ 
sign changes in $\psi(x, \xi^*) - s^*$. As $\psi(x, \xi^*)$ is 
a polynomial of even degree, also the number of (proper)
sign changes has to be even, and they occur at
$a_1 > \dots > a_{2r}$, say, $r \leq q$. Moreover, 
for $0 < \alpha < 1$, $\XX^*$ is a proper subset of $\R$ 
and, thus, there must be at least one sign change, 
$r \geq 1$. Finally, as the
leading coefficient of $\psi(x, \xi^*)$ is positive, 
$\psi(x, \xi^*)$ gets larger than $s^*$ for $x \to \pm \infty$ 
and, hence, the outmost intervals $[a_1, \infty)$ and 
$(- \infty, a_{2r}]$ are included in the support $\XX^*$ of
$\xi^*$. By the interlacing property of intervals with positive and
negative sign for  $\psi(x, \xi^*) - s^*$, the result follows from the 
conditions on the $ D $-optimal subsampling design $ \xi^{*} $
in Theorem~\ref{th:equivalence}.
\end{proof}

\begin{proof}[Proof of Theorem~\ref{th:equivariant}]
	First note that for any $\mu$ and $ \sigma >0 $, 
	the location-scale transformation $ z = \sigma x + \mu $ 
	is conformable with the regression function $ \f(x) $, 
	i.\,e.\ there exists a non-singular matrix 
	$ {\bf Q} $ such that $ \f(\sigma x + \mu) = {\bf Q}\f(x) $ for all $x$. 
	Then, for any design $ \xi $ bounded by $ f_{X}(x) $, 
	the design $ \zeta $ has density 
	$  f_{\zeta}(z) = \frac{1}{\sigma} f_{\xi}(\frac{z - \mu}{\sigma}) $ 
	bounded by 
	$  f_{Z}(z) = \frac{1}{\sigma} f_{X}(\frac{z - \mu}{\sigma}) $.
	Hence, by the transformation theorem for measure integrals, 
	it holds that
	\begin{align*}
		\M(\zeta) 
		&= \int \f(z)\f(z)^{\top} \zeta(\diff z) \\
		&= \int \f(\sigma x + \mu)\f(\sigma x + \mu)^{\top} \xi(\diff x) \\
		&= \int {\bf Q} \f(x)\f(x)^{\top} {\bf Q}^{\top}\xi(\diff x) \\
		&= {\bf Q} \M(\xi) {\bf Q}^{\top}.
	\end{align*}
	Therefore 
	$ \det(\M(\zeta)) = \det({\bf Q})^{2} \det(\M(\xi))$. 
	Thus $ \xi^{*} $ maximizes the $ D $-criterion over the set of 
	subsampling designs bounded by $ f_{X}(x) $ if and only if $ \zeta_{*} $ 
	maximizes the $ D $-criterion over the set of subsampling designs bounded 
	by $ f_{Z}(z) $.
\end{proof}

\begin{proof}[Proof of Corollary~\ref{cor:opt-design-degree-q-symm}]
	The checkerboard structure of the information matrix $\M(\xi^*)$
	carries over to its inverse $\M(\xi^*)^{-1}$.
	Hence, the sensitivity function $\psi(x, \xi^*)$ is an even polynomial,
	which has only non-zero coefficients for even powers of $x$,
	and is thus symmetric with respect to $0$, 
	i.\,e.\ $\psi(-x, \xi^*) = \psi(x, \xi^*)$.
	Accordingly, also the roots of $\psi(x, \xi^*) =s^*$ are
	symmetric with respect to $0$.
\end{proof}

\begin{proof}[Proof of Theorem~\ref{Theo:normal}]
	In view of the shape~\eqref{eq:density-quadratic} of the density
	and by Corollary~\ref{cor:opt-design-degree-q-symm},
	the tails are included in the optimal subsampling design
	such that $a < \infty$.
	
	Next, we consider the symmetric design $\xi^\prime$ which is supported 
	only on the tails and which will be the optimal subsampling design when $b = 0$.
	This design has density 
	$f_{\xi^\prime}(x) = \1_{(-\infty, -a] \cup [a, \infty)}(x) f_X(x)$
	with $a = z_{1 - \alpha/2}$ for given $\alpha$. 
	The information matrix $\M(\xi^\prime)$ is of the form~\eqref{eq:info-quadratic}
	with relevant entries
	\begin{align*}
		m_2(\xi^{\prime}) &= \alpha + \sqrt{2/\pi} a \exp(-a^{2}/2), \\
		m_4(\xi^{\prime}) &= 3 m_{2}(\xi^{\prime}) + \sqrt{2/\pi} a^{3} \exp(-a^{2}/2) \, .
	\end{align*}
	For the sensitivity function~\eqref{eq:sensitivity-quadratic}, we have
	\begin{align*}
		\psi(0, \xi^{\prime}) 
		&= \frac{\alpha m_{4}(\xi^{\prime})}{\alpha m_{4}(\xi^{\prime}) - m_{2}(\xi^{\prime})^{2}}
	\end{align*}
	and 
		\begin{align*}
		\psi(a, \xi^{\prime}) 
		&= \frac{\alpha m_{4}(\xi^{\prime})}{\alpha m_{4}(\xi^{\prime}) - m_{2}(\xi^{\prime})^{2}}
		- \frac{\alpha 2 m_{2}(\xi^{\prime}) a^{2}}{\alpha m_{4}(\xi^{\prime}) - m_{2}(\xi^{\prime})^{2}} \\
		&\qquad\qquad+ \frac{\alpha a^2}{m_{2}(\xi^{\prime})}
		+ \frac{\alpha^2 a^{4}}{\alpha m_{4}(\xi^{\prime}) - m_{2}(\xi^{\prime})^{2}} \, .
	\end{align*}
	Let $c(\alpha) = \psi(0, \xi^{\prime}) - \psi(a, \xi^{\prime})$ be the difference
	between the values of the sensitivity function at $x = 0$ and $x = a$,
	then
	\begin{equation}
		\label{eq:c-normal-quadratic}
		c(\alpha) = \alpha a^{2}\left(\frac{2m_2(\xi^{\prime})}{\alpha m_{4}(\xi^{\prime}) 
			- m_{2}(\xi^{\prime})^{2}} - \frac{a^2\alpha}{\alpha m_{4}(\xi^{\prime}) 
			- m_{2}(\xi^{\prime})^{2}} - \frac{1}{m_2(\xi^{\prime})}\right) \, .
	\end{equation}
	$c(\alpha)$ is continuous in $\alpha$ and does not have any roots in $(0,1)$. 
	Further, it can be checked that $c(0.1) > 0$, say.
	Thus $c(\alpha) > 0$ 
	which means that $\psi(0, \xi^{\prime}) > \psi(a, \xi^{\prime})$ 
	for all $\alpha $.
	Hence, by Theorem~\ref{th:equivalence},
	the subsampling design $\xi^{\prime}$ cannot be optimal
	and, as a consequence, the optimal subsampling design $\xi^*$ has support
	on three proper intervals with $ b > 0 $ for all $ \alpha $.
\end{proof}

\begin{proof}[Proof of Theorem~\ref{th:unif-alpha-to-0}]	
	Let
	\begin{align*}
		u(\alpha) &= 45 - 15\alpha + 15\alpha^2 - 45\alpha^3 + 20\alpha^4 \\&\qquad\:\,- 4\sqrt{5}\sqrt{45\alpha^2 - 90\alpha^3 + 90\alpha^4 - 75\alpha^5 + 57\alpha^6 - 27\alpha^7 + 5\alpha^8}
	\end{align*}
	and
	\begin{align*}
		v(\alpha) &= 180 (1 - \alpha) \, .
	\end{align*}
	Then
	\begin{equation*}
		b(\alpha) 
			= \Biggl( \frac{u(\alpha)}{v(\alpha)} \Biggr)^{1/2} - \frac{1}{2} (1 - \alpha) \, .
	\end{equation*}
	We have $u(0) = 45$, $v(0) = 180$, 
	and $b(\alpha)$ can be continuously extended to $b(0) = 0$ at $\alpha = 0$. 
	The derivative of $b$ is given by
	\begin{align}
		\label{eq:unifb-derivative}
		b'(\alpha) = \frac{1}{2} 
			+ \frac{1}{2} \frac{u'(\alpha)v(\alpha) - u(\alpha)v'(\alpha)}{v(\alpha)^2} \sqrt{\frac{v(\alpha)}{u(\alpha)}} \, ,
	\end{align}
	where
	\begin{align}
		u'(\alpha) &= -15 + 30\alpha - 135\alpha^2 + 80\alpha^3 - w(\alpha), \\
		v'(\alpha) &= -180 ,
	\end{align}
	and
	\begin{align*}
		w(\alpha) &= 2 \sqrt{5} 
			\frac{90 - 270\alpha + 360\alpha^2 - 375\alpha^3 + 342 \alpha^4 - 189\alpha^5 
				+ 40\alpha^6} 
				{\sqrt{45 - 90\alpha + 90\alpha^2 - 75\alpha^3 + 57\alpha^4 - 27\alpha^5 + 5\alpha^6}} .
	\end{align*}
	We have $v'(0) = -180$. 
	To determine $u'(0)$ we 
	note that $w(0) = 60$ and thus $u'(0) = -75$.
	Hence, also the derivative $b'(\alpha)$ can be continuously extended at
	$\alpha = 0$ and the value for $b'(0)$ can be obtained by plugging in 
	the values of $u(0)$, $v(0)$, $u'(0)$, and $v'(0)$ 
	into formula~\eqref{eq:unifb-derivative},
	\begin{align*}
		b'(0) = \frac{1}{2} + \frac{1}{2} \frac{-75 \cdot 180 + 45 \cdot 180}{180^2} \sqrt{\frac{180}{45}} 
			= \frac{1}{3} \, .
	\end{align*}
	Finally, we note that $b(\alpha)/\alpha$ is the percentage of mass on 
	the interior interval $[-b(\alpha), b(\alpha)]$
	and that $\lim_{\alpha \to 0} b(\alpha)/\alpha$ 
	is the derivative $b'(0)$ of $ b(\alpha) $ at $ \alpha = 0 $.
	Hence, the percentage of mass on the interior interval tends to $b'(0) = 1/3$ 
	when the subsampling proportion $\alpha$ goes to $0$.
\end{proof}

\begin{proof}[Proof of Theorem~\ref{Theo:t5}]
	The proof will follow the idea of the proof of 
	Theorem~\ref{Theo:normal}. 
	For $\alpha \in (0, 1)$, 
	we consider the symmetric design $\xi^\prime$ which is supported 
	only on the tails and which will be the optimal subsampling design when $b = 0$.
	This design has density 
	$f_{\xi^\prime}(x) = \1_{(-\infty, -a] \cup [a, \infty)(x)} f_X(x)$
	with $a = t_{5, 1 - \alpha/2}$. 
	The relevant entries of the information matrix $\M(\xi^\prime)$ are
	\begin{align*}
		m_2(\xi^{\prime}) 
			&= \frac{5}{3\pi} \left(\pi - \frac{2\sqrt{5}a(a^2 - 5)}{(a^2 + 5)^2} - 2 \arctan(a/\sqrt{5}) \right), 
		\\
		m_4(\xi^{\prime}) 
			&= \frac{25}{3\pi} 
				\left(3\pi + \frac{10\sqrt{5}a(a^2 +3)}{(a^2 + 5)^2} - 6 \arctan(a/\sqrt{5}) \right).
	\end{align*}
	The sensitivity function $\psi(x, \xi^\prime)$ and 
	the difference $c(\alpha) = \psi(0, \xi^{\prime}) - \psi(a, \xi^{\prime})$  
	between the values of the sensitivity function at $x = 0$ and $x = a$
	are defined as for the normal distribution with the above moments
	$m_2(\xi^{\prime})$ and $m_4(\xi^{\prime})$ related to the $t$-distribution inserted.
	The function $c(\alpha)$ defined by~\eqref{eq:c-normal-quadratic} 
	then looks as shown in Figure~\ref{Figure:Quad-t5-c}. 
	\begin{figure}[h]
		\centering
		\includegraphics[width=.6\textwidth]{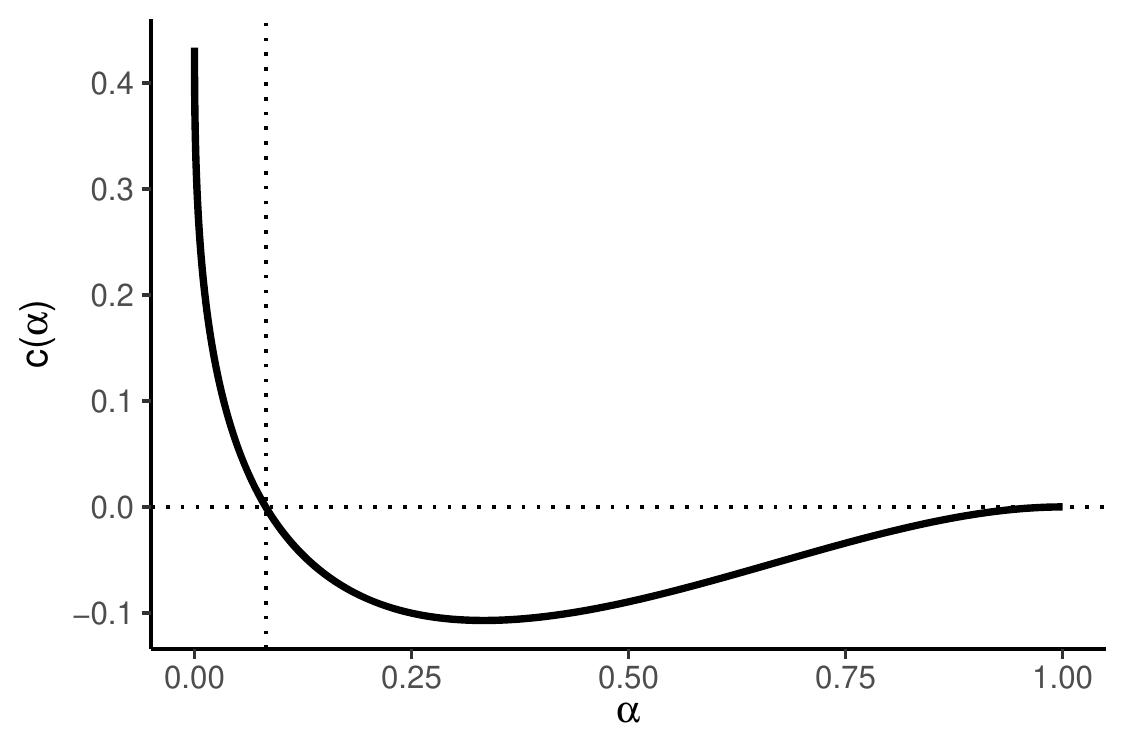}
		\caption{Difference $ c(\alpha) = \psi(0, \xi^{\prime}) - \psi(a, \xi^{\prime})$ (solid) 
			for the case of a $ t $-distributed covariate with $5$ degrees of freedom}
		\label{Figure:Quad-t5-c}
	\end{figure}
	The vertical dotted line indicates the position of the 
	critical value $ \alpha^{*} \approx 0.082065$,
	where the curve of the function $c(\alpha)$ intersects 
	the horizontal dotted line indicating $ c = 0 $.

	Thus for $ \alpha < \alpha^{*} \approx 0.082065 $ we have 
	$ \psi(0,\xi^{\prime}) > \psi(a,\xi^{\prime}) $ 
	and the design $\xi^\prime$ cannot be optimal by Theorem~\ref{th:equivalence}.
	In this situation, an inner interval has to be included 
	in the optimal subsampling design $\xi^*$
	with $b > 0$.

	Conversely, for $ \alpha \geq \alpha^{*} \approx 0.082065 $ we have that 
	$ \psi(0,\xi^{\prime}) \leq \psi(a,\xi^{\prime}) $.
	Hence, the design $\xi^\prime$ is optimal by Theorem~\ref{th:equivalence},
	and no inner interval has to be added 
	to the optimal subsampling design $\xi^* = \xi^\prime$ ($b = 0$).
\end{proof}

\newpage
	\section*{Acknowledgments}
	The work of the first author was supported by the Deutsche Forschungsgemeinschaft 
	(DFG, German Research Foundation) within the Research Training Group ``MathCoRe'' 
	under grant GRK 2297. 
	\bibliographystyle{plainnat}  
	\bibliography{ref}

\end{document}